\documentclass[12pt]{amsart}
\usepackage[all]{xy}

\usepackage{graphicx} 

\usepackage{amsmath}
\usepackage{amssymb}
\usepackage{amsfonts}

\newcommand{\Cdb}{\mbox{$\mathbb{C}$}}
\newcommand{\Ddb}{\mbox{$\mathbb{D}$}}

\newcommand{\Rdb}{\mbox{$\mathbb{R}$}}

\newcommand{\A}{\mbox{${\mathcal A}$}}

\newcommand{\F}{\mbox{${\mathcal F}$}}

\renewcommand{\P}{\mbox{${\mathcal P}$}}

\newcommand{\U}{\mbox{${\mathcal U}$}}

\newcommand{\norm}[1]{\Vert#1\Vert}
\newcommand{\bignorm}[1]{\bigl\Vert#1\bigr\Vert}
\newcommand{\Bignorm}[1]{\Bigl\Vert#1\Bigr\Vert}
\newcommand{\biggnorm}[1]{\biggl\Vert#1\biggl\Vert}

\newcommand{\rnorm}[1]{\Vert#1\Vert_{r}}

\newcommand{\Rad}{{\rm Rad}}

\newcommand{\HI}{H^\infty}

\newtheorem{theorem}{Theorem}[section]
\newtheorem{lemma}[theorem]{Lemma}
\newtheorem{corollary}[theorem]{Corollary}
\newtheorem{proposition}[theorem]{Proposition}

\theoremstyle{remark}
\newtheorem{remark}[theorem]{\bf Remark}
\theoremstyle{definition}

\numberwithin{equation}{section}

\begin{document}

\title[]{Maximal theorems and square functions for analytic operators on
$L^p$-spaces}

\author{Christian Le Merdy, Quanhua Xu}
\address{Laboratoire de Math\'ematiques\\ Universit\'e de  Franche-Comt\'e
\\ 25030 Besan\c con Cedex\\ France}
\email{clemerdy@univ-fcomte.fr}

\address{Laboratoire de Math\'ematiques\\ Universit\'e de  Franche-Comt\'e
\\ 25030 Besan\c con Cedex\\ France}
\email{qxu@univ-fcomte.fr}

\date{\today}

\thanks{The authors are both supported by the research program ANR-06-BLAN-0015}

\begin{abstract} 
Let $T\colon L^p(\Omega)\to L^p(\Omega)$ be a contraction, with $1<p<\infty$, 
and assume that $T$ is analytic, that is, $\sup_{n\geq 1}n\norm{T^n-T^{n-1}}\,<\infty\,$. 
Under the assumption that $T$ is positive (or contractively regular), 
we establish the boundedness of various Littlewood-Paley square functions associated with $T$. 
In particular we show that $T$ satisfies 
an estimate $\bignorm{\bigl(\sum_{n=1}^{\infty} n^{2m-1}\bigl\vert T^n(T-I)^m(x)
\bigr\vert^2\bigr)^{\frac{1}{2}}}_p\,\lesssim \norm{x}_p$ for any integer $m\geq 1$.
As a consequence we show maximal inequalities of the form 
$\bignorm{\sup_{n\geq 0}\, (n+1)^m\bigl\vert T^n(T-I)^m(x)
\bigr\vert}_p\,\lesssim\, \norm{x}_p$, for any integer $m\geq 0$. 
We prove similar results in the context of noncommutative $L^p$-spaces. We 
also give analogs of these maximal inequalities for bounded analytic semigroups, 
as well as applications to $R$-boundedness properties.
\end{abstract}

\maketitle


\bigskip

\section{Introduction.}
Let $(\Omega,\mu)$ be a measure space, let $1<p<\infty$ and let $T\colon L^p(\Omega)\to L^p(\Omega)$ be
a positive contraction. Then Akcoglu's Theorem \cite{A} asserts that $T$ satisfies a
maximal ergodic inequality,
\begin{equation}\label{1Ergodic}
\Bignorm{\sup_{n\geq 0}\,\frac{1}{n+1}\,\Bigl\vert  
\sum_{k=0}^{n} T^k(x)\Bigr\vert}_p\,\lesssim \norm{x}_p,\qquad x\in L^p(\Omega).
\end{equation}
A well-known question is to determine which operators satisfy a
stronger maximal inequality, 
\begin{equation}\label{1Max0}
\Bignorm{\sup_{n\geq 0}\,\vert  T^n(x)\vert}_p\,\lesssim  \norm{x}_p,\qquad x\in L^p(\Omega).
\end{equation}
In this paper we show that this holds true provided that 
$T$ is analytic, that is, there exists a constant $K\geq 0$ such that 
$$
n\norm{T^{n}-T^{n-1}}\leq K
$$ 
for any $n\geq 1$ (see Section 2 for some background).
More generally, we show that for any integer $m\geq 0$,
analytic positive contractions
$T\colon L^p(\Omega)\to L^p(\Omega)$
satisfy a maximal inequality
\begin{equation}\label{1Maxm}
\Bignorm{\sup_{n\geq 0}\, (n+1)^m\bigl\vert T^n(T-I)^m(x)\bigr\vert}_p\,\lesssim\, 
\norm{x}_p,\qquad x\in L^p(\Omega).
\end{equation}
Note that for any $m\geq 1$, the sequence of operators $(T^n(T-I)^m)_{n\geq 0}$
appearing here is the $m$-th order discrete derivative of the original sequence $(T^n)_{n\geq 0}$.
The proofs of these inequalities rely on the boundedness of certain 
discrete Littlewood-Paley square functions of independent interest that we establish in Section 3. 
In particular we will show that 
for $T$ as above, we have an estimate
\begin{equation}\label{1LP1}
\Bignorm{\Bigl(\sum_{n=1}^{\infty} n\,\bigl\vert T^n(x)-T^{n-1}(x)
\bigr\vert^2\Bigr)^{\frac{1}{2}}}_p\lesssim \norm{x}_p,\qquad 
x\in L^p(\Omega).
\end{equation}
These maximal theorems and square function estimates extend Stein's famous results \cite{S1,S2} 
which show that (\ref{1Max0}),
(\ref{1Maxm}) and (\ref{1LP1}) hold true in the case when $T$ acts as a contraction 
$L^q(\Omega) \to L^q(\Omega)$ for
any $1\leq q\leq \infty$ and its $L^2$-realization is a positive selfadjoint operator.

\bigskip
Let $M$ be a von Neumann algebra equipped with a normal semifinite faithful trace 
and for any $1\leq p\leq \infty$, let
$L^p(M)$ be the associated noncommutative $L^p$-space. Let $T\colon M\to M$ be a 
positive contraction whose restriction to $L^1(M)\cap M$ extends to a
contraction $T\colon L^1(M)\to L^1(M)$. Recall that in this case, 
$T$ actually extends to a contraction $L^q(M)\to L^q(M)$ for any $1\leq q\leq \infty$.
It is shown in \cite{JX} that $T$ satisfies a noncommutative analog 
of (\ref{1Ergodic}). In the latter paper, a large part of Stein's work mentioned above is also 
transfered to the noncommutative setting. Indeed it is shown that 
if the $L^2$-realization $T\colon L^2(M)\to L^2(M)$ is a positive selfadjoint operator, then 
for any $1<p<\infty$, $T$ satisfies noncommutative analogs 
of (\ref{1Max0}) and (\ref{1Maxm}). This is generalized in \cite{Bek} under an  
appropriate condition on the numerical range of $T\colon L^2(M)\to L^2(M)$. 
We extend these results
by showing that for any $1<p<\infty$, the noncommutative analogs of 
(\ref{1Max0}) and (\ref{1Maxm}) hold true provided that $T\colon L^p(M)\to L^p(M)$ is 
merely analytic (which is a much weaker assumption).

\bigskip
Besides investigating the behaviour of operators and their powers (discrete semigroups), 
we consider continuous semigroups $(T_t)_{t\geq 0}$, both in the commutative
and in the noncommutative settings. The continuous analog of the maximal inequality 
(\ref{1Max0}) reads as follows:
\begin{equation}\label{1SG}
\Bignorm{\sup_{t>0}\bigl\vert T_t(x)\bigr\vert}_p \lesssim\norm{x}_p.
\end{equation}
We prove that such an estimate holds true whenever $(T_t)_{t\geq 0}$ is a bounded analytic 
semigroup on $L^p(\Omega)$ (with $1<p<\infty$) such that $T_t\colon L^p(\Omega)
\to L^p(\Omega)$ is a positive contraction for any $t\geq 0$. Likewise we show 
that the noncommutative analog of (\ref{1SG}) holds true whenever $(T_t)_{t\geq 0}$
is a semigroup of positive contractions on $L^q(M)$ for any $1\leq q\leq\infty$ and 
$(T_t)_{t\geq 0}$ is a bounded analytic semigroup on $L^p(M)$ (with $1<p<\infty$).
These results both extend Stein's classical maximal theorem \cite{S1,S2} 
for semigroups and its recent noncommutative counterpart from \cite{JX}.
Finally we extend some results from \cite[Chapter 5]{JLX} concerning
$R$-boundedness in the noncommutative setting.

\bigskip 
In the above presentation and later on in the paper, $\lesssim$ stands for an inequality up to 
a constant which may depend on $T$ and $m$, but not on $x$.

\medskip
\section{Preliminaries.}
An operator $T\colon L^p(\Omega)\to L^p(\Omega)$ is called regular if there is a constant 
$C\geq 0$ such that 
$$
\bignorm{\sup_{k\geq 1}\vert T(x_k)\vert}_p\,\leq\, C\bignorm{\sup_{k\geq 1}\vert x_k\vert}_p
$$
for any finite sequence $(x_k)_{k\geq 1}$ in $L^p(\Omega)$. Then we let $\norm{T}_r$ denote the 
smallest $C$ for which this holds. The set of all regular operators on $L^p(\Omega)$ is a
vector space on which $\norm{\ }_r$ is a norm. We say that $T$ is contractively regular if
$\norm{T}_r\leq 1$. Clearly any positive operator
$T$ is regular and $\norm{T}_r=\norm{T}$ in this case. Thus all statements given for 
contractively regular operators apply to positive contractions. It is well-known that conversely,
$T$ is regular with $\norm{T}_r\leq C$ if and only if there is a positive operator
$S\colon L^p(\Omega)\to L^p(\Omega)$ with $\norm{S}\leq C$, such that 
$\vert T(x)\vert \leq S(\vert x\vert)$ for any $x\in L^p(\Omega)$ (see \cite[Chap. 1]{MN}). 
Furthermore, $T$ is contractively regular if
$T$ acts as a contraction $L^q(\Omega)\to L^q(\Omega)$ for any
$1\leq q \leq \infty$.

\bigskip
We recall some definitions and simple facts about sectorial
operators and analyticity. Throughout we let $X$ denote an arbitrary (complex) Banach space and  
we let $B(X)$ denote the algebra of all bounded operators on $X$. Next 
for any angle $\omega\in(0,\pi)$, we introduce
$$
\Sigma_{\omega}\,=\,\bigl\{z\in \Cdb^*\, :\, \vert{\rm Arg}(z)\vert<\omega\bigr\},
$$
the open sector of angle $2\omega$ around $(0,\infty)$. 

Let $A\colon D(A)\subset X \to X$ be a (possibly unbounded)
closed linear operator, with dense domain $D(A)$. We let $\sigma(A)$ denote the 
spectrum of $A$ and for any $\lambda\in\Cdb\setminus\sigma(A)$, we let 
$R(\lambda,A)=(\lambda-A)^{-1}$ denote the corresponding resolvent operator.
We say that $A$ is sectorial if there exists an angle $\theta\in(0,\pi)$
such that $\sigma(A)$ is contained in the closed sector
$\overline{\Sigma_{\theta}}$ and
$$
(S)_\theta\qquad\qquad\qquad\qquad 
\exists\, K\geq 0\quad \big\vert \quad  \vert \lambda \vert \norm{R(\lambda,A)}\leq K,\qquad \lambda \in
\Cdb\setminus \overline{\Sigma_{\theta}}.\qquad\qquad\qquad\qquad\qquad
$$
Then we let $\omega(A)$ be the infimum of all $\theta$ such that $(S)_\theta$ holds, and this 
real number is called the type of $A$.
It is well-known that if $(S)_\theta$ holds true for some $\theta\in(0,\pi)$, then there  
exists $\varepsilon>0$ such that $(S)_{\theta-\varepsilon}$ holds true as well.
Thus,
\begin{equation}\label{2Sector1}
(S)_{\theta}\ \Longleftrightarrow\ \omega(A)<\theta.
\end{equation}

Let $(T_t)_{t\geq 0}$ be a bounded strongly continuous semigroup on $X$.
We call it a bounded analytic semigroup if there exists a positive angle 
$\alpha\in\bigl(0,\frac{\pi}{2}\bigr)$ and a 
bounded analytic family $z\in\Sigma_\alpha\mapsto  T_z\in B(X)$ extending $(T_t)_{t>0}$.
Let $-A$ be the infinitesimal generator of $(T_t)_{t\geq 0}$. Analyticity has two 
classical characterizations in terms of that operator. First, $(T_t)_{t\geq 0}$ is a bounded
analytic semigroup if and only if $T_t(X)\subset D(A)$ for any $t>0$ and
there exists a constant $K\geq 0$ such that 
$\norm{tAT_t}\leq K$ for any $t>0$.
Note here that since $T_t=e^{-tA}$, we have 
\begin{equation}\label{2Derivative}
tAT_t\,=\, -t\,\frac{\partial}{\partial t}\bigl(T_t\bigr),\qquad t>0.
\end{equation}
Second, $(T_t)_{t\geq 0}$ is a bounded
analytic semigroup if and only if $A$ is sectorial and $\omega(A)<\frac{\pi}{2}$. According  
to (\ref{2Sector1}), this is also equivalent to saying that $A$ satisfies $(S)_{\frac{\pi}{2}}$.
We refer e.g. to \cite{Go,Pa} for proofs and complements on semigroups.

\bigskip
We will make a crucial use of $H^\infty$-calculus and square functions for sectorial operators.
Here are the basic notions and results which will be needed.
For more information, we refer  e.g. to
\cite{CDMY, KaW1, KW, LM}.

For any $\theta\in (0,2\pi)$, we define
$$
\HI(\Sigma_\theta)  = \{f\colon\Sigma_\theta \to \Cdb\,\vert\,
f\text{ is analytic and bounded}\}.
$$
This is a Banach algebra with the norm 
$$
\norm{f}_{\HI(\Sigma_\theta)} =
\sup\bigl\{\vert f(\lambda)\vert\, :\, \lambda\in \Sigma_\theta\}.
$$
Then let $\HI_0(\Sigma_\theta)\subset \HI(\Sigma_\theta)$ be the subalgebra of all
$f$ for which there exist two constants $s,C>0$ such that
$$
\vert f(\lambda)\vert \leq C\, 
\min\bigl\{\vert \lambda\vert^s,\, \vert \lambda\vert^{-s}\bigr\},\qquad\lambda\in \Sigma_\theta.
$$
For any sectorial operator $A$, for any
$\theta\in (\omega(A),\pi)$ and for any $f\in
\HI_0(\Sigma_\theta)$, we define
$$
f(A)\, =\, \frac{1}{2\pi i}\,\int_{\Gamma_\gamma} f(\lambda)
R(\lambda,A)\, d\lambda,
$$
where $\omega(A)<\gamma<\theta$ and $\Gamma_\gamma$ is the
boundary $\partial\Sigma_\gamma$ oriented counterclockwise. This integral is
well-defined, its
definition does not depend on $\gamma$ and the resulting mapping
$f\mapsto f(A)$ is an algebra homomorphism from
$\HI_0(\Sigma_\theta)$ into $B(X)$. We say that $A$ has a bounded
$\HI(\Sigma_\theta)$ functional calculus if the latter homomorphism is
bounded, that is, there exists a constant $C>0$ such that
$$
\norm{f(A)}\leq C \norm{f}_{\HI(\Sigma_\theta)},\qquad f\in
\HI_0(\Sigma_\theta).
$$

Consider now the specific case when $X=L^p(\Omega)$, with $1<p<\infty$. On such a space,
Cowling, Doust, McIntosh and Yagi have proved a remarkable
equivalence result between the boundedness of $\HI$ functional calculus 
and certain square function estimates. In particular they
established the following key result.

\begin{proposition}\label{2SFE} \cite{CDMY}
Let $A$ be a sectorial operator on $L^p(\Omega)$ and assume that there exists
$\theta_0\in (0,\pi)$ such that $A$ admits a bounded $\HI(\Sigma_\theta)$ functional calculus 
for any $\theta\in(\theta_0,\pi)$. Then for any $\theta\in(\theta_0,\pi)$ and any
$\varphi\in\HI_0(\Sigma_\theta)$, there exists a constant $C\geq 0$ such that 
\begin{equation}\label{2SFEbis}
\Bignorm{\Bigl(\int_{0}^{\infty}\bigl\vert \varphi(tA)x\bigr\vert^2\, 
\frac{dt}{t}\,\Bigr)^{\frac{1}{2}}}_p\, \leq\, C\norm{x}_p,\qquad 
x\in L^p(\Omega).
\end{equation}
\end{proposition}

\bigskip 
Let us now turn to discrete semigroups. Let $T\in B(X)$. We say that $T$ is power
bounded if the set
\begin{equation}\label{2Power}
\P_T=\{T^n\, :\, n\geq 0\}
\end{equation}
is bounded. Then we say that $T$ is analytic if moreover the set
\begin{equation}\label{2Analytic}
\A_T=\bigl\{n(T^{n}-T^{n-1})\,  :\,n\geq 1\bigr\}
\end{equation}
is bounded. This notion of discrete analyticity goes back to 
\cite{CSC}. Since $(T^{n}-T^{n-1})_{n\geq 1}$ is the `discrete derivative' of the sequence
$(T^{n})_{n\geq 0}$, we can regard $n(T^{n}-T^{n-1})$ as a discrete analog of 
$t\frac{\partial}{\partial t}(T_t)$. In view of (\ref{2Derivative}), the boundedness of (\ref{2Analytic}) is 
therefore a natutal discrete analog of the boundedness of $\{tAT_t\, :\, t>0\}$.

The most important result concerning discrete analyticity is perhaps the following characterization:
an operator $T\colon X\to X$ is
power bounded and analytic if and only if
\begin{equation}\label{2Ritt}
\sigma(T)\subset \overline{\Ddb}\qquad\hbox{and}\qquad
\bigl\{(\lambda-1)R(\lambda,T)\,:\,\vert\lambda\vert>1\bigr\}\ \hbox{is bounded}.
\end{equation}
This property is called the `Ritt condition'. The key argument 
for this characterization is due to O. Nevanlinna \cite{N}, however we refer
to \cite{Ly,NZ} for a complete proof and complements. Let us gather a few observations which will be used
later on in the paper. First we note that (\ref{2Ritt}) implies that
\begin{equation}\label{2Spectrum}
\sigma(T)\subset\Ddb\cup\{1\}.
\end{equation}
Indeed, $\norm{R(\lambda, T)}\geq d(\lambda,\sigma(T))^{-1}$ for any $\lambda\notin\sigma(T)$. Second, 
(\ref{2Ritt}) implies the existence of a constant $K\geq 0$
such that $\vert \lambda-1 \vert\norm{R(\lambda,T)}\leq K$ whenever ${\rm Re}(\lambda)>1.$ This means 
that 
$$
A=I-T
$$
satisfies $(S)_{\frac{\pi}{2}}$. According to (\ref{2Sector1}), this implies that
$A$ is a sectorial operator of type $<\frac{\pi}{2}$. Hence 
\begin{equation}\label{2Sector2}
\exists\,\theta\in\bigl(0,\tfrac{\pi}{2}\bigr)\ \big\vert\ \sigma(T)\subset \, 1-\overline{\Sigma_\theta}.
\end{equation}
In this case, 
the bounded analytic semigroup $(T_t)_{t\geq 0}$ generated by $-A$ is given by 
\begin{equation}\label{2Tt}
T_t=e^{-t}e^{tT},\qquad t\geq 0.
\end{equation} 

\bigskip
We now recall the definition of $R$-boundedness (see \cite{BG,CPSW}).
Let $(\varepsilon_k)_{k\geq 1}$ be a sequence of independent
Rademacher variables on some probability space $\Omega_0$. Let
$\Rad(X)\subset L^2(\Omega_0;X)$ be the closure of ${\rm
Span}\{\varepsilon_k\otimes x\, :\, k\geq 1,\ x\in X\}$ in the Bochner space
$L^2(\Omega_{0};X)$. Thus for any finite family $x_1,\ldots,x_n$
in $X$, we have
$$
\Bignorm{\sum_k \varepsilon_k\otimes x_k}_{\Rad(X)} \,=\,
\Bigr(\int_{\Omega_0}\Bignorm{\sum_k \varepsilon_k(s)\,
x_k}_{X}^{2}\,ds\,\Bigr)^{\frac{1}{2}}.
$$
By definition, a set $\F\subset B(X)$ is $R$-bounded if there is
a constant $C\geq 0$ such that for any finite families
$T_1,\ldots, T_n$ in $\F$, and any $x_1,\ldots,x_n$ in $X$, we have
$$
\Bignorm{\sum_k \varepsilon_k\otimes T_k (x_k)}_{\Rad(X)}\,\leq\, C\,
\Bignorm{\sum_k \varepsilon_k\otimes x_k}_{\Rad(X)}.
$$
Obviously any $R$-bounded set is bounded and if $X$  is isomorphic 
to a Hilbert space, then  all bounded subsets of $B(X)$ are automatically
$R$-bounded. However if $X$ is not isomorphic to a
Hilbert space, then $B(X)$ contains bounded subsets which are not
$R$-bounded \cite[Prop. 1.13]{AB}.

Let $(\Omega,\mu)$ be a measure space and let $1<p<\infty$. 
Then $\Rad(L^p(\Omega))\approx L^p(\Omega;\ell^2)$. 
Hence a set $\F\subset B(L^p(\Omega))$ is $R$-bounded if and only if we have an estimate
$$
\Bignorm{\Bigl(\sum_k\bigl\vert T_k(x_k)\bigr\vert^2\Bigr)^{\frac{1}{2}}}_p\,\leq \, C\,
\Bignorm{\Bigl(\sum_k \vert  x_k \vert^2\Bigr)^{\frac{1}{2}}}_p
$$
for finite families $(T_k)_k$ in $\F$ and $(x_k)_k$ in $X$.

We shall now consider these general definitions for specific sets of operators.
Let $(T_t)_{t\geq 0}$ be a bounded analytic semigroup on $X$.
We say that this is an $R$-bounded analytic semigroup if there exists a positive angle $\alpha>0$ 
such that $\{ T_z\, :\, z\in\Sigma_\alpha\}$ is $R$-bounded. 
It was observed in \cite{W1} that this holds true if and only if the two sets
$$
\bigl\{ T_t\, :\, t>0\bigr\}\qquad\hbox{and}\qquad\bigl\{ tAT_t\, :\, t>0 \bigr\}
$$
are $R$-bounded. 

Accordingly we will say that an operator $T\in B(X)$ is an $R$-analytic power bounded
operator if the two sets  
$\P_T$ and $\A_T$ from (\ref{2Power}) and (\ref{2Analytic}) are $R$-bounded.

The above notions of $R$-analyticity were introduced by Weis \cite{W1} for the continuous 
case and Blunck \cite{Bl1} for the discrete one. In both cases they played a crucial role
in the solution of maximal regularity problems on UMD Banach spaces, 
see the above papers for more information. $R$-boundedness for sectorial
operators is also a key tool for various questions regarding 
$H^\infty$ functional calculus, see in particular \cite{KaW1,KW,JLX}.

The next result is well-known to specialists.

\begin{proposition}\label{2Weis} 
Let $(T_t)_{t\geq 0}$ be a bounded analytic semigroup on $L^p(\Omega)$, with $1<p<\infty$,
and assume
that $\rnorm{T_t}\leq 1$ for any $t\geq 0$. Let $-A$ be the generator of $(T_t)_{t\geq 0}$.
Then there exists $\theta\in \bigl(0,\frac{\pi}{2}\bigr)$ such  that $A$ admits a bounded $\HI(\Sigma_\theta)$
functional calculus.
\end{proposition}

\begin{proof}
By \cite{Du} (see also \cite[Thm. 4.13]{LM}), the operator
$A$ admits a bounded $\HI(\Sigma_\theta)$ functional calculus for any $\theta>\frac{\pi}{2}$. 
On the other hand, it follows from \cite[Section 4]{W2} that 
$(T_t)_{t\geq 0}$ is an $R$-bounded analytic semigroup. 
Applying \cite[Prop. 5.1]{KaW1} we deduce the result. 
\end{proof}

\bigskip 
We end this section with a few notation. For any complex number $a$ and any $r>0$, 
we will let $D(a,r)$ denote the open disc of center $a$ and radius $r$. We let 
$\Ddb=D(0,1)$ be the usual unit disc. Also we let 
$\P$ denote the algebra of complex polynomials in one variable.

\medskip
\section{Square functions on $L^p(\Omega)$.}
Throughout the next two sections we let $(\Omega,\mu)$ be a measure space and 
we fix some $1<p<\infty$. We will establish general square function estimates 
for analytic contractively regular operators on $L^p(\Omega)$ (see Theorem \ref{3DiscreteSF} below). 

We will need the following elementary fact.

\begin{lemma}\label{3EasySFE}
Let $\U\subset\Cdb$ be an open set and let $\Gamma\subset \U$ be a compact $C^1$-curve.
Let $\varphi\colon \U\to B(L^p(\Omega))$ be an analytic function. Then there exists a contant $C\geq 0$ such that
$$
\Bignorm{\Bigl(\int_{\Gamma}\bigl\vert \varphi(\lambda)x\bigr\vert^2\, \vert d\lambda\vert\,\Bigr)^{\frac{1}{2}}}_p\,\leq C\norm{x}_p
$$
for any $x\in L^p(\Omega)$.
\end{lemma}

\begin{proof}
Let $0< r\leq d(\Gamma,\U^c)/3$. Write $\Gamma$ as the juxtaposition of
$C^1$-curves $\Gamma_1,\ldots,\Gamma_N$ of length $<r$. Then for each $j=1,\ldots, N$, choose $\lambda_j\in\Gamma_j$
and set 
$$
C_j=\sup\{\norm{\varphi(\lambda)}\, :\, \lambda\in  D(\lambda_j,2r)\}.
$$ 
Let
\begin{equation}\label{3Taylor}
\varphi(\lambda)\,=\,\sum_{k=0}^{\infty} c_{jk}\,(\lambda-\lambda_j)^k
\end{equation}
be the Taylor expansion of $\varphi$ about $\lambda_j$. Then $\norm{c_{jk}}\leq C_j/(2r)^k$ by Cauchy's inequalities.
Any $\lambda\in\Gamma_j$ satisfies (\ref{3Taylor}) hence we have
$$
\Bignorm{\Bigl(\int_{\Gamma_j}\bigl\vert \varphi(\lambda)x\bigr\vert^2\, 
\vert d\lambda\vert\,\Bigr)^{\frac{1}{2}}}_p\,\leq \, \sum_{k=0}^{\infty}
\Bignorm{\Bigl(\int_{\Gamma_j}\bigl\vert c_{jk}(x)\, (\lambda-\lambda_j)^k\bigr\vert^2\, 
\vert d\lambda\vert\,\Bigr)^{\frac{1}{2}}}_p.
$$
However for any $k\geq 0$, we have
$$
\Bignorm{\Bigl(\int_{\Gamma_j}\bigl\vert c_{jk}(x)\, (\lambda-\lambda_j)^k\bigr\vert^2\, 
\vert d\lambda\vert\,\Bigr)^{\frac{1}{2}}}_p\, =\, \bignorm{c_{jk}(x)}_p\,
\Bigl(\int_{\Gamma_j}\vert\lambda-\lambda_j\vert^{2k}\,\vert d\lambda\vert\,\Bigr)^{\frac{1}{2}}
$$
and $\vert\lambda-\lambda_j\vert\leq r$ for any $\lambda\in\Gamma_j$. Thus
$$
\Bignorm{\Bigl(\int_{\Gamma_j}\bigl\vert c_{jk}(x)\, (\lambda-\lambda_j)^k\bigr\vert^2\, 
\vert d\lambda\vert\,\Bigr)^{\frac{1}{2}}}_p\,
\leq\, \norm{c_{jk}}\,\norm{x}_p\,\vert\Gamma_j\vert r^k\,
\leq\, \norm{x}_p\,\vert\Gamma_j\vert \,\frac{C_j}{2^k}\,.
$$
Consequently,
$$
\Bignorm{\Bigl(\int_{\Gamma_j}\bigl\vert \varphi(\lambda)x\bigr\vert^2\, 
\vert d\lambda\vert\,\Bigr)^{\frac{1}{2}}}_p\,\leq \,
2C_j\norm{x}_p\vert\Gamma_j\vert.
$$
Since
$$
\Bignorm{\Bigl(\int_{\Gamma}\bigl\vert \varphi(\lambda)x\bigr\vert^2\, \vert d\lambda\vert\,\Bigr)^{\frac{1}{2}}}_p\,=
\,\sum_{j=1}^N \Bignorm{\Bigl(\int_{\Gamma_j}\bigl\vert \varphi(\lambda)x\bigr\vert^2\, \vert d\lambda\vert\,\Bigr)^{\frac{1}{2}}}_p,
$$
we obtain the result with $C=2\max\{C_1,\ldots,C_N\}\,\vert\Gamma\vert$.
\end{proof}

For any $\gamma\in \bigl(0,\frac{\pi}{2}\bigr)$, let 
$$
B_\gamma= \,\bigl\{z\in  \bigl(1+\Sigma_{\pi-\gamma}\bigr)^c\, :\, 
\vert z\vert \leq \sin\gamma\ \hbox{or}\ {\rm Re}(z)\geq \sin^2\gamma\bigr\}.
$$
Alternatively, $B_\gamma$ is the convex hull of $1$ and the disc $D(0,\sin\gamma)$.

\begin{figure}[ht] 
\vspace*{2ex} 
\begin{center} 
\includegraphics[scale=0.4]{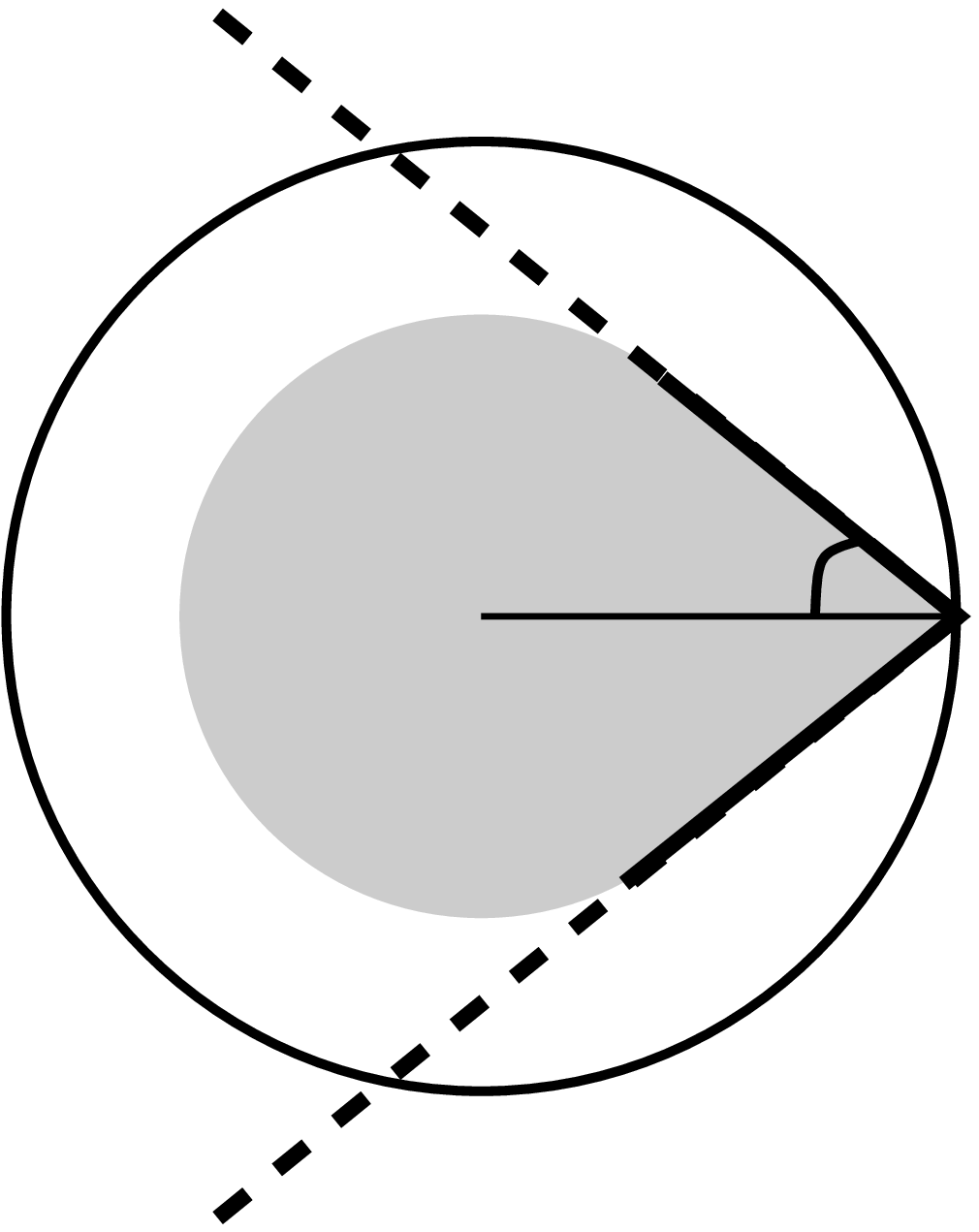}
\begin{picture}(0,0) 
\put(-2,65){{\footnotesize $1$}} 
\put(-68,65){{\footnotesize $0$}} 
\put(-32,81){{\footnotesize $\gamma$}} 
\put(-55,85){{\small $B_\gamma$}} 
\end{picture} 
\end{center} 
\caption{\label{f1}} 
\end{figure}

Following usual terminology, these sets will be called `Stolz domains' in the sequel.
We will use the fact that for any $\gamma\in \bigl(0,\frac{\pi}{2}\bigr)$, there 
exists a constant $C_\gamma$ such that
\begin{equation}\label{3Stolz}
\frac{\vert 1-z\vert}{1-\vert z\vert}\,\leq C_\gamma,\qquad z\in B_\gamma.
\end{equation}

Let $N\geq 1$ be an integer
and let $[F_{ij}]$ be an $N\times N$ matrix of polynomials, that is, $F_{i,j}$ belongs 
to $\P$ for any $1\leq i,j\leq N$. Then for any $\gamma\in \bigl(0,\frac{\pi}{2}\bigr)$, we set
$$
\bignorm{[F_{ij}]}_\gamma\,=\,\sup\Bigl\{\bignorm{[F_{ij}(z)]}_{M_N}\, :\, z\in B_\gamma\Bigr\}.
$$

\begin{proposition}\label{3Matrix}
Let $T\colon L^p(\Omega)\to L^p(\Omega)$ be any analytic contractively regular operator. Then there exists an
angle $\gamma\in \bigl(0,\frac{\pi}{2}\bigr)$ and a constant $C\geq 1$ satisfying the following
property. For any $N\geq 1$, for any $N\times N$ matrix $[F_{ij}]$ of polynomials 
and for any
$x_1,\ldots, x_N$ in $L^p(\Omega)$, we have
\begin{equation}\label{3Matrix1}
\Bignorm{\Bigl(\sum_{i=1}^{N}\Bigl\vert\sum_{j=1}^{N} F_{ij}(T)x_j\Bigr\vert^2\Bigr)^{\frac{1}{2}}}_p\,\leq\, C
\bignorm{[F_{ij}]}_\gamma\,\Bignorm{\Bigl(\sum_{j=1}^{N}\vert x_j\vert^2\Bigr)^{\frac{1}{2}}}_p.
\end{equation}
\end{proposition}

\begin{proof}
We let $p'=p/(p-1)$ be the conjugate number of $p$. Let 
$
A=I-T
$
and let $(T_t)_{t\geq 0}$ be the semigroup defined by (\ref{2Tt}), whose generator  is $-A$.
We noticed in Section 2 that this is a bounded analytic semigroup. 
Furthermore for any $t\geq 0$, we have
$$
\norm{T_t}_r = e^{-t}\bignorm{e^{tT}}_r\leq  e^{-t} e^{t\norm{T}_r}\leq 1.
$$
Hence by Proposition \ref{2Weis}, $A$ admits a bounded $\HI(\Sigma_{\theta_0})$ functional calculus 
for some $\theta_0<\frac{\pi}{2}$.
By (\ref{2Spectrum}) and (\ref{2Sector2}), there exists 
$\gamma_0\in\bigl[\theta_0,\frac{\pi}{2}\bigr)$ such that 
$\sigma(T)\subset B_{\gamma_0}$. Equivalently,
$$
\sigma(A)=1-\sigma(T) \subset 1 -B_{\gamma_0}.
$$
We now fix $\gamma\in\bigl(\gamma_0,\frac{\pi}{2}\bigr)$. 
Then we let $L_\gamma$ be the boundary of $1-B_\gamma$ oriented counterclockwise.

\begin{figure}[ht] 
\vspace*{2ex} 
\begin{center} 
\includegraphics[scale=0.45]{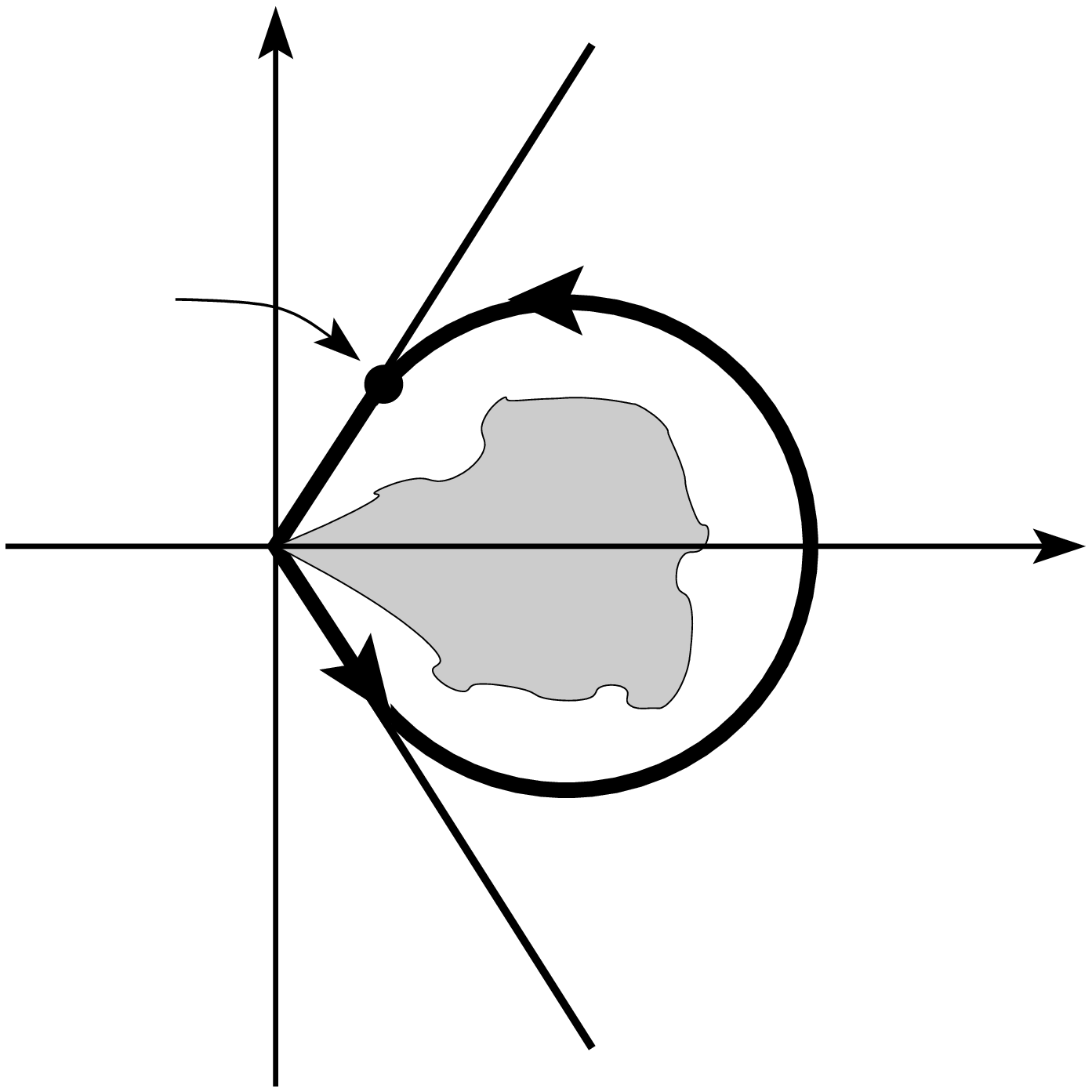}
\begin{picture}(0,0) 
\put(-148,79){{\footnotesize $0$}} 
\put(-97,174){{\small $\gamma$}} 
\put(-75,132){{\small $L_\gamma$}} 
\put(-100,98){{\small $\sigma(A)$}} 
\put(-201,131){{\small $\cos(\gamma)e^{i\gamma}$}} 
\end{picture} 
\end{center} 
\caption{ \label{f2}} 
\end{figure}

We claim that we have estimates
\begin{equation}\label{3Pf1}
\Bignorm{\Bigl(\int_{L_{\gamma}}\bigl\vert A^{\frac{1}{2}}(\lambda-A)^{-1}
x\bigr\vert^2\,\vert d\lambda\vert\,\Bigr)^{\frac{1}{2}}}_p\,
\lesssim\,\norm{x}_p,\qquad x\in L^p(\Omega),
\end{equation}
and
\begin{equation}\label{3Pf2}
\Bignorm{\Bigl(\int_{L_{\gamma}}\bigl\vert 
A^{*\frac{1}{2}}(\lambda+A^*)^{-1} y\bigr\vert^2\,\vert d\lambda\vert\,\Bigr)^{\frac{1}{2}}}_{p'}\,
\lesssim\, \norm{y}_{p'},\qquad y\in L^{p'}(\Omega).
\end{equation}
Recall that we let $\Gamma_\gamma$ denote the boundary of $\Sigma_\gamma$ oriented counterclockwise.
Thus the contour $L_{\gamma}$ is the juxtaposition of a part $L_{\gamma,1}$ of $\Gamma_{\gamma}$
and the curve $L_{\gamma,2}$ going from  $\cos(\gamma) e^{-i\gamma}$ to $\cos(\gamma) e^{i\gamma}$
counterclockwise along the circle of center $1$ and radius $\sin\gamma$. Obviously we have
\begin{align*}
\Bignorm{\Bigl(\int_{L_{\gamma}}\bigl\vert A^{\frac{1}{2}}(\lambda-A)^{-1}
x\bigr\vert^2\,\vert d\lambda\vert\,\Bigr)^{\frac{1}{2}}}_p \, & =\,
\Bignorm{\Bigl(\int_{L_{\gamma, 1}}\bigl\vert A^{\frac{1}{2}}(\lambda-A)^{-1}
x\bigr\vert^2\,\vert d\lambda\vert\,\Bigr)^{\frac{1}{2}}}_p\\&  \quad +
\Bignorm{\Bigl(\int_{L_{\gamma, 2}}\bigl\vert A^{\frac{1}{2}}(\lambda-A)^{-1}
x\bigr\vert^2\,\vert d\lambda\vert\,\Bigr)^{\frac{1}{2}}}_p\,\\ &\leq\,
\Bignorm{\Bigl(\int_{\Gamma_{\gamma}}\bigl\vert A^{\frac{1}{2}}(\lambda-A)^{-1}
x\bigr\vert^2\,\vert d\lambda\vert\,\Bigr)^{\frac{1}{2}}}_p\\&  \quad +
\Bignorm{\Bigl(\int_{L_{\gamma, 2}}\bigl\vert A^{\frac{1}{2}}(\lambda-A)^{-1}
x\bigr\vert^2\,\vert d\lambda\vert\,\Bigr)^{\frac{1}{2}}}_p\,
\end{align*}
Since $L_{\gamma,2}\cap \sigma(A)=\emptyset$, Lemma \ref{3EasySFE} ensures that we can control the 
last integral by a constant times $\norm{x}_p$. Hence to prove (\ref{3Pf1}), it suffices to  prove an estimate
\begin{equation}\label{3Pf3}
\Bignorm{\Bigl(\int_{\Gamma_{\gamma}}\bigl\vert A^{\frac{1}{2}}(\lambda-A)^{-1}
x\bigr\vert^2\,\vert d\lambda\vert\,\Bigr)^{\frac{1}{2}}}_p\,
\lesssim\,  \norm{x}_p,\qquad x\in L^p(\Omega).
\end{equation}
Likewise, to prove (\ref{3Pf2}), it suffices to prove an estimate
\begin{equation}\label{3Pf4}
\Bignorm{\Bigl(\int_{\Gamma_{\gamma}}\bigl\vert 
A^{*\frac{1}{2}}(\lambda+A^*)^{-1} y\bigr\vert^2\,\vert d\lambda\vert\,\Bigr)^{\frac{1}{2}}}_{p'}\,
\lesssim\,  \norm{y}_{p'},\qquad y\in L^{p'}(\Omega).
\end{equation}
Consider $\theta_0<\theta<\gamma< \frac{\pi}{2}\,$, and define two functions $\varphi,\psi
\in \HI_0(\Sigma_\theta)$ by letting 
$$
\varphi(z)=\,\frac{z^{\frac{1}{2}}}{e^{i\gamma} -z} \qquad\hbox{and}\qquad
\psi(z)=\,\frac{z^{\frac{1}{2}}}{e^{-i\gamma} -z}\,.
$$ 
For any $x\in L^p(\Omega)$, we have
\begin{align*}
\Bignorm{\Bigl(\int_{0}^{\infty}\bigl\vert \varphi(tA)x\bigr\vert^2\, 
\frac{dt}{t}\,\Bigr)^{\frac{1}{2}}}_p\,
& =\, \Bignorm{\Bigl(\int_{0}^{\infty}\bigl\vert 
A^{\frac{1}{2}}(e^{i\gamma} -tA)^{-1}x\bigr\vert^2\, dt\,\Bigr)^{\frac{1}{2}}}_p\\ & =
\,
\Bignorm{\Bigl(\int_{0}^{\infty}\bigl\vert A^{\frac{1}{2}}(te^{i\gamma} -A)^{-1}x\bigr\vert^2\, dt\,\Bigr)^{\frac{1}{2}}}_p.
\end{align*}
Likewise,
$$
\Bignorm{\Bigl(\int_{0}^{\infty}\bigl\vert \psi(tA)x\bigr\vert^2\, \frac{dt}{t}\,\Bigr)^{\frac{1}{2}}}_p\,
=\,
\Bignorm{\Bigl(\int_{0}^{\infty}\bigl\vert A^{\frac{1}{2}}(te^{-i\gamma} -A)^{-1}x\bigr\vert^2\, dt\,\Bigr)^{\frac{1}{2}}}_p.
$$
Hence
$$
\Bignorm{\Bigl(\int_{\Gamma_{\gamma}}\bigl\vert A^{\frac{1}{2}}(\lambda-A)^{-1}
x\bigr\vert^2\,\vert d\lambda\vert\,\Bigr)^{\frac{1}{2}}}_p \,=\, 
\Bignorm{\Bigl(\int_{0}^{\infty}\bigl\vert \varphi(tA)x\bigr\vert^2\, \frac{dt}{t}\,\Bigr)^{\frac{1}{2}}}_p\,
\,+\,
\Bignorm{\Bigl(\int_{0}^{\infty}\bigl\vert \psi(tA)x\bigr\vert^2\, \frac{dt}{t}\,\Bigr)^{\frac{1}{2}}}_p\,.
$$
Applying Proposition \ref{2SFE} to $\varphi$ and $\psi$, we deduce the estimate (\ref{3Pf1}). 
Now note that $A^*$ also admits a bounded $\HI(\Sigma_{\theta_0})$ 
functional calculus (see e.g. \cite{CDMY} for this duality principle). Hence arguing as above 
with the two functions 
$$
z\mapsto\,\frac{z^{\frac{1}{2}}}{e^{i\gamma} +z}\,\qquad\hbox{and}\qquad
z\mapsto\,\frac{z^{\frac{1}{2}}}{e^{-i\gamma} +z}\,,
$$
we get (\ref{3Pf2}). 

The estimates (\ref{3Pf1}) and (\ref{3Pf2}) can be formally strengthened as follows. 
There is a constant $C\geq 0$ such that for any integer $N\geq 1$, we have
\begin{equation}\label{3Pf5}
\Bignorm{\Bigl(\int_{L_{\gamma}}\sum_{j=1}^{N} \bigl\vert A^{\frac{1}{2}}(\lambda-A)^{-1}
x_j\bigr\vert^2\,\vert d\lambda\vert\,\Bigr)^{\frac{1}{2}}}_p\,
\leq\, C\, \Bignorm{\Bigl(\sum_{j=1}^{N}\vert x_j\vert^2\Bigr)^{\frac{1}{2}}}_p
\end{equation}
for any $x_1,\ldots,x_N$ in $L^p(\Omega)$ and similarly,
\begin{equation}\label{3Pf6}
\Bignorm{\Bigl(\int_{L_{\gamma}}  \sum_{j=1}^{N} \bigl\vert
A^{*\frac{1}{2}}(\lambda+A^*)^{-1} y_i \bigr\vert^2\,\vert d\lambda\vert\,\Bigr)^{\frac{1}{2}}}_{p'}\,
\leq\,  C\, \Bignorm{\Bigl(\sum_{i=1}^{N}\vert y_i\vert^2\Bigr)^{\frac{1}{2}}}_{p'}
\end{equation}
for any $y_1,\ldots,y_N$ in $L^{p'}(\Omega)$. Indeed (\ref{3Pf5}) (resp. (\ref{3Pf6})) can be deduced from 
(\ref{3Pf1}) (resp. (\ref{3Pf2})) by applying Khintchine's inequality and
Fubini's Theorem. The argument is similar to the one in the proof of \cite[Lemma 5.4]{LLL} so we omit it.

In the sequel, we let $\P_0\subset \P$ be the space of polynomials vanishing at $0$.
The function $\lambda\mapsto \lambda(\lambda-A)^{-1}$ 
is well-defined and bounded on $L_\gamma\setminus\{0\}$, and the same is true for 
$\lambda\mapsto f(\lambda)(\lambda-A)^{-1}$ whenever
$f\in\P_0$. It therefore follows from the Dunford functional calculus that 
$$
f(A)=\,\frac{1}{2\pi i}\,\int_{L_{\gamma}} 
f(\lambda)(\lambda-A)^{-1}\, d\lambda
$$
for any $f\in\P_0$.
Likewise,
$$
0=\,\frac{1}{2\pi i}\,\int_{L_{\gamma}} f(\lambda)(\lambda +A)^{-1}\, d\lambda
$$
for any $f\in\P_0$. Hence
$$
f(A)  = \,\frac{1}{2\pi i}\,\int_{L_{\gamma}} f(\lambda)\bigl((\lambda-A)^{-1} -
(\lambda +A)^{-1}\bigr)
\, d\lambda\,,
$$
that is,
\begin{equation}\label{3Int}
f(A) = \,\frac{1}{\pi i}\,\int_{L_{\gamma}} f(\lambda) A(\lambda-A)^{-1}
(\lambda +A)^{-1}
\, d\lambda\,.
\end{equation}
Let $N\geq 1$ be an integer, let $[F_{ij}]$ be an $N\times N$ matrix of polynomials, and let $x_1,\ldots,x_N$ be in 
$L^p(\Omega)$. For any $i,j=1,\ldots,N$, we set $f_{ij}(\lambda)=F_{ij}(1-\lambda)$, so that $F_{ij}(T)=f_{ij}(A)$.
Also we assume that $F_{ij}(1)=0$, so that $f_{ij}\in\P_0$. For any $y_1,\ldots, y_N$ in $L^{p'}(\Omega)$, we have
\begin{align*}
\sum_{i,j}\bigl\langle f_{ij}(A)x_j,y_i\bigr\rangle\, & =\,
\frac{1}{\pi i}\,\int_{L_{\gamma}}\sum_{i,j}f_{ij}(\lambda)
\bigl\langle A(\lambda-A)^{-1}
(\lambda +A)^{-1} x_j,y_i\bigr\rangle\, d\lambda\\
& =\,
\frac{1}{\pi i}\,\int_{L_{\gamma}}\sum_{i,j}f_{ij}(\lambda)
\bigl\langle A^{\frac{1}{2}}
(\lambda -A)^{-1} x_j,A^{*\frac{1}{2}}(\lambda +A^*)^{-1}y_i\bigr\rangle\, d\lambda
\end{align*}
by (\ref{3Int}). Applying Cauchy-Schwarz and H\"older's inequalities,
we deduce that 
\begin{align*}
\biggl\vert \sum_{i,j}\bigl\langle f_{ij}(A)x_j,y_i\bigr\rangle\biggr\vert\,  \leq & \,
\frac{1}{\pi}\,\Bignorm{\Bigl(\int_{L_\gamma}\sum_i\Bigl\vert\sum_j f_{ij}(\lambda) 
A^{\frac{1}{2}}(\lambda -A)^{-1} x_j\Bigr\vert^2\,\vert d\lambda\vert\,\Bigr)^{\frac{1}{2}}}_p\\
&\quad \times
\Bignorm{\Bigl(\int_{L_\gamma}\sum_i\bigl\vert
A^{*\frac{1}{2}}(\lambda +A^*)^{-1} y_i\bigr\vert^2\,\vert d\lambda\vert\,\Bigr)^{\frac{1}{2}}}_{p'}\,.
\end{align*}
Furthermore, 
$$
\Bignorm{\Bigl(\int_{L_\gamma}\sum_i\Bigl\vert\sum_j f_{ij}(\lambda) 
A^{\frac{1}{2}}(\lambda -A)^{-1} x_j\Bigr\vert^2\,\vert d\lambda\vert\,\Bigr)^{\frac{1}{2}}}_p\,
$$
is less than or equal to
$$
\Bignorm{\Bigl(\int_{L_\gamma}\bignorm{[f_{ij}(\lambda)]}_{M_N}^{2}\sum_j \bigl\vert 
A^{\frac{1}{2}}(\lambda -A)^{-1} x_j\bigr\vert^2\,\vert d\lambda\vert\,\Bigr)^{\frac{1}{2}}}_p,
$$
which in turn is less than or equal to
$$
\sup\Bigl\{\bignorm{[f_{ij}(\lambda)]}_{M_N}\, :\, \lambda\in L_{\gamma}\Bigr\}\,
\Bignorm{\Bigl(\int_{L_\gamma}\sum_j \bigl\vert 
A^{\frac{1}{2}}(\lambda -A)^{-1} x_j\bigr\vert^2\,\vert d\lambda\vert\,\Bigr)^{\frac{1}{2}}}_p.
$$
Now recall that $F_{ij}(T)=f_{ij}(A)$ and note that $\sup\bigl\{\norm{[f_{ij}(\lambda)]}_{M_N}\, :\, \lambda\in L_{\gamma}\bigr\}$ is less than or equal 
to $\norm{[F_{ij}}_\gamma$. Appealing to (\ref{3Pf5}) and (\ref{3Pf6}), 
we therefore obtain an estimate 
$$
\biggl\vert \sum_{i,j}\bigl\langle F_{ij}(T)x_j,y_i\bigr\rangle\biggr\vert\, \lesssim \,
\bignorm{[F_{ij}}_\gamma\,
\Bignorm{\Bigl(\sum_j\vert x_j\vert^2\Bigr)^{\frac{1}{2}}}_p\, 
\Bignorm{\Bigl(\sum_i\vert y_i\vert^2\Bigr)^{\frac{1}{2}}}_{p'}.
$$
Passing to the supremum over all $y_1,\ldots,y_N$ in $L^{p'}(\Omega)$ such that
$\bignorm{\bigl(\sum_i\vert y_i\vert^2\bigr)^{\frac{1}{2}}}_{p'}\leq 1$, we finally obtain 
(\ref{3Matrix1})  in the case when all $F_{ij}$'s vanish at $1$.

The general case follows at once. Indeed for an arbitrary matrix $[F_{ij}]$ of polynomials,
write $\widetilde{F}_{ij}=F_{ij}-F_{ij}(1)$. Then 
$$
\bignorm{[F_{ij}(1)]}_{M_N}\leq \bignorm{[F_{ij}]}_{\gamma}\qquad\hbox{and}\qquad
\bignorm{[\widetilde{F}_{ij}]}_{\gamma}\leq 2\bignorm{[F_{ij}]}_{\gamma}.
$$
Thus if (\ref{3Matrix1}) holds true for $[\widetilde{F}_{ij}]$ and a certain
constant $C$,  we deduce that
\begin{align*}
\Bignorm{\Bigl(\sum_{i=1}^{N}\Bigl\vert
\sum_{j=1}^{N} F_{ij}(T)x_j\Bigr\vert^2\Bigr)^{\frac{1}{2}}}_p\, & 
\leq\,\Bignorm{\Bigl(\sum_{i=1}^{N}\Bigl\vert
\sum_{j=1}^{N} \widetilde{F}_{ij}(T)x_j\Bigr\vert^2\Bigr)^{\frac{1}{2}}}_p
\, +\, 
\Bignorm{\Bigl(\sum_{i=1}^{N}\Bigl\vert
\sum_{j=1}^{N} F_{ij}(1)x_j\Bigr\vert^2\Bigr)^{\frac{1}{2}}}_p\\
& 
\leq\,\bigl(2C +1\bigr)\bignorm{[F_{ij}]}_{\gamma}\,
\Bignorm{\Bigl(\sum_{j=1}^{N} \vert x_j \vert^2\Bigr)^{\frac{1}{2}}}_p.
\end{align*}
\end{proof}

\begin{theorem}\label{3DiscreteSF}
Let $T\colon L^p(\Omega)\to L^p(\Omega)$ be an analytic contractively regular operator. 
\begin{enumerate}
\item [(1)] There exists an
angle $\gamma\in \bigl(0,\frac{\pi}{2}\bigr)$ and a constant $C\geq 1$ such that for any 
sequence $(F_n)_{n\geq 1}$ of polynomials and any $x\in L^p(\Omega)$,
\begin{equation}\label{3DiscreteSF1}
\Bignorm{\Bigl(\sum_{n=1}^{\infty}\bigl\vert F_n(T)x\bigr\vert^2\Bigr)^{\frac{1}{2}}}_p\,\leq\, C\,\norm{x}_p\,\sup\Bigl\{ 
\Bigl(\sum_{n=1}^{\infty}\bigl\vert F_n(z)\bigr\vert^2\Bigr)^{\frac{1}{2}}\, :\, z\in B_\gamma\Bigr\}.
\end{equation}
\item [(2)] For any integer $m\geq 1$, there is an estimate
\begin{equation}\label{3DiscreteSF2}
\Bignorm{\Bigl(\sum_{n=0}^{\infty} (n+1)^{2m-1}\bigl\vert T^{n}(T-I)^m(x)
\bigr\vert^2\Bigr)^{\frac{1}{2}}}_p\,\lesssim \,\norm{x}_p.
\end{equation}
\end{enumerate}
\end{theorem} 

\begin{proof} We apply Proposition \ref{3Matrix} to $T$ and we thus obtain
$\gamma\in \bigl(0,\frac{\pi}{2}\bigr)$ for which (\ref{3Matrix1}) holds true. Let
$(F_n)_{n\geq 1}$ be any sequence of polynomials. We get (\ref{3DiscreteSF1}) by  applying
(\ref{3Matrix1}) to the column matrix 
$$
\left[\begin{array}{cccc} F_1 & 0 & \cdots & 0 \\ \vdots & \vdots & \ & \vdots \\ F_N  & 0 &\cdots & 0\end{array}\right]
$$
for any $N\geq 1$ and then by passing to the limit when $N\to\infty$.

To prove part (2), we fix $m\geq 1$, we set 
$$
F_n(z)= n^{m-\frac{1}{2}} z^{n-1}(z-1)^m
$$
for any $n\geq 1$,
and we aim at applying (\ref{3DiscreteSF1}) to this sequence.
For any $z\in\Ddb$, we have
\begin{align*}
\sum_{n=1}^{\infty}\bigl\vert F_n(z)\bigr\vert^2 \, 
& = \,
\sum_{n=1}^{\infty} 
n^{2m-1}\vert z\vert^{2(n-1)}\vert z-1\vert^{2m}\\ 
& \leq \,
\vert 1-z\vert^{2m} \sum_{n=0}^{\infty} 
(n+1)(n+2)\cdots(n+2m-1) \vert z\vert^{2n}\\
& \leq \,
\vert 1-z \vert^{2m} \,\frac{1}{\bigl(1-\vert z\vert^2\bigr)^{2m}}
\\
& \leq \,\biggl(\frac{\vert 1-z\vert}{1-\vert z\vert}\,\biggr)^{2m}.
\end{align*}
This upper bound is bounded on $B_\gamma$ by (\ref{3Stolz})
hence (\ref{3DiscreteSF2}) now follows from part (1).
\end{proof}

Note that (\ref{1LP1}) corresponds to (\ref{3DiscreteSF2}) for $m=1$.

\begin{remark}\label{3Remark1} Consider $T$ as in Theorem \ref{3DiscreteSF}. 
We will establish additional estimates, which are all consequences of the above theorem. 

\smallskip (1) By the Mean Ergodic Theorem, we 
have a direct sum decomposition
$$
L^p(\Omega)\,=\,N(I-T)\oplus\overline{R(I-T)},
$$
where $N(\cdotp)$ and $R(\cdotp)$ denote the kernel and the range, respectively.
Let $P\colon L^p(\Omega)\to L^p(\Omega)$ be the projection onto 
$\overline{R(I-T)}$ with respect to this decomposition. Then for any $m\geq 1$, we have
an estimate 
\begin{equation}\label{3Reverse}
\norm{P(x)}_p\,\lesssim\, 
\Bignorm{\Bigl(\sum_{n=0}^{\infty} (n+1)^{2m-1}\bigl\vert T^{n}(T-I)^m(x)
\bigr\vert^2\Bigr)^{\frac{1}{2}}}_p
\end{equation}
on $L^p(\Omega)$. In other words, the estimate (\ref{3DiscreteSF2}) can be reversed on 
$\overline{R(I-T)}$.

Let us prove (\ref{3Reverse}) for $m=1$, the other cases being similar. We start from the identity
$$
\sum_{n=0}^{\infty}(n+1)z^{2n}(1-z^2)^2\, =1, \qquad z\in\Ddb.
$$
It implies that for any $0<r<1$, we have
$$
\sum_{n=0}^{\infty}(n+1)(rT)^{2n}(rT+I)^2(rT-I)^2\, =I.
$$
Let $x\in L^p(\Omega)$ and $y\in L^{p'}(\Omega)$. Set $y_r=(rT^*+I)^2y$ for any $r$.
From the above identity, we get
$$
\langle x,y\rangle\,=\,
\sum_{n=0}^{\infty}(n+1)\bigl\langle (rT)^{n}(rT-I) x, 
(rT^*)^n (rT^*-I) y_r\bigr\rangle.
$$
Hence
$$
\bigl\vert \langle x,y\rangle  \bigr\vert\,
\leq\, 
\Bignorm{\Bigl(\sum_{n=0}^{\infty} (n+1)\bigl\vert (rT)^n(rT-I)x\bigr\vert^2\Bigr)^{\frac{1}{2}}}_{p}\,
\Bignorm{\Bigl(\sum_{n=0}^{\infty} (n+1)\bigl\vert (rT^*)^n(rT^*-I)y_r\bigr\vert^2\Bigr)^{\frac{1}{2}}}_{p'}.
$$
The operator $T^*\colon L^{p'}(\Omega)\to L^{p'}(\Omega)$ is analytic and contractively 
regular, hence satisfies the first part of Theorem \ref{3DiscreteSF}. Moreover $\norm{y_r}_{p'}\leq 4\norm{y}_{p'}$
for any $r$. Hence we can control the second factor in the right handside of the above inequality
by $\norm{y}_{p'}$, up to
a constant not depending on $r$. We deduce that 
$$
\bigl\vert \langle x,y\rangle  \bigr\vert\,
\lesssim\, 
\Bignorm{\Bigl(\sum_{n=0}^{\infty} (n+1)\bigl\vert (rT)^n(rT-I)x\bigr\vert^2\Bigr)^{\frac{1}{2}}}_p\,\norm{y}_{p'}
$$
uniformly in $r$. Taking the supremum over all $y\in L^{p'}(\Omega)$ with $\norm{y}_{p'}\leq 1$, we obtain a uniform
estimate
$$
\norm{x}_p\,\lesssim\,\Bignorm{\Bigl(\sum_{n=0}^{\infty} (n+1)r^{2n}\bigl\vert T^n(rT-I)x\bigr\vert^2\Bigr)^{\frac{1}{2}}}_p,\qquad x\in L^p(\Omega),\ 0<r<1.
$$
Now assume that $x\in R(I-T)$, i.e. $x=(T-I)\widetilde{x}$ for some $\widetilde{x}$ in $L^p(\Omega)$.
Applying (\ref{3DiscreteSF2}) to $\widetilde{x}$ (with $m=1$), we see that
the sequence $\bigl((n+1)^{\frac{1}{2}} T^n(x)\bigr)_{n\geq 0}$ belongs to $L^p(\ell^2)$.
Consequently, the sequence $\bigl((n+1)^{\frac{1}{2}}(rT-I) T^n(x)\bigr)_{n\geq 0}$ belongs to $L^p(\ell^2)$ as well
for any $0<r<1$ and this family of sequences 
tends to $\bigl((n+1)^{\frac{1}{2}} (T-I)T^n(x)\bigr)_{n\geq 0}$
when $r\to 1$. We deduce that 
$$
\Bignorm{\Bigl(\sum_{n=0}^{\infty} (n+1)\bigl\vert T^n(rT-I)x\bigr\vert^2\Bigr)^{\frac{1}{2}}}_p
\,\longrightarrow \,
\Bignorm{\Bigl(\sum_{n=0}^{\infty} (n+1)\bigl\vert T^n(T-I)x\bigr\vert^2\Bigr)^{\frac{1}{2}}}_p
$$
when $r\to 1$, and hence that
\begin{equation}\label{3Reverse2}
\norm{x}_p\,\lesssim 
\,\Bignorm{\Bigl(\sum_{n=0}^{\infty} (n+1)\bigl\vert T^n(T-I)x\bigr\vert^2\Bigr)^{\frac{1}{2}}}_p.
\end{equation}
This establishes (\ref{3Reverse}) for the elements of $R(I-T)$.

To complete the proof, set 
$$
\Lambda_m\,=\,\frac{1}{m+1}\sum_{k=1}^{m} (I-T^k)
$$
for any integer $m\geq 0$. Then $\Lambda_m\to P$ pointwise when $m\to\infty$.
Let $x$ be an arbitrary element of $L^p(\Omega)$.
Applying (\ref{3Reverse2}) with $\Lambda_m(x)$ in the place of $x$
and letting $m\to\infty$, we obtain 
the desired estimate (\ref{3Reverse}).

\smallskip (2) 
For any $m\geq 1$, $T$ 
satisfies the following estimate
\begin{equation}\label{3DiscreteSF3}
\Bignorm{\Bigl(\sum_{n=1}^{\infty} n\,\bigl\vert (n+1)^{m}T^{n}(T-I)^{m} (x)\, -\, 
n^{m}T^{n-1}(T-I)^{m} (x)\bigr\vert^2\Bigr)^{\frac{1}{2}}}_p\,\lesssim\,\norm{x}_p,
\end{equation}
that we record here for further use in Section 4. 

For its proof it will be convenient to set 
\begin{equation}\label{3Delta}
\Delta_{n}^{m} = T^{n}(T-I)^{m}\qquad\hbox{and}\qquad B_{n}^{m} =(n+1)^m \Delta_{n}^{m}
\end{equation}
for any integers $m,n\geq 0$. We fix some $m\geq 1$ and $x\in L^p(\Omega)$. Then we have
\begin{align*}
B_n^m(x)-B_{n-1}^m(x)\, &=\ \bigl((n+1)^m T - n^m\bigr) T^{n-1}(T-I)^m x\\
&=\, (n+1)^m T^{n-1} (T-I)^{m+1} x\, +\, \bigl((n+1)^m -n^m\bigr) T^{n-1}(T-I)^m x
\end{align*}
for any $n\geq 1$. Consequently,
\begin{align*}
\bigl\vert 
B_n^m(x)-B_{n-1}^m(x)\bigr\vert^2\,&\leq\,2\bigl( (n+1)^{2m} \bigl\vert
\Delta_{n-1}^{m+1}(x)\bigr\vert^2\, +\, 
\bigl((n+1)^m -n^m\bigr)^2\bigl\vert
\Delta_{n-1}^{m}(x)\bigr\vert^2\bigr)\\
&\lesssim\, n^{2m}  \bigl\vert
\Delta_{n-1}^{m+1}(x)\bigr\vert^2\, +\, n^{2(m-1)} \bigl\vert
\Delta_{n-1}^{m}(x)\bigr\vert^2.
\end{align*}
Summing up, we obtain that 
$$
\sum_{n=1}^{\infty} n\,\bigl\vert 
B_n^m(x)-B_{n-1}^m(x)\bigr\vert^2\,\lesssim\, 
\sum_{n=1}^{\infty}  n^{2m+1}  \bigl\vert
\Delta_{n-1}^{m+1}(x)\bigr\vert^2\, +\, 
\sum_{n=1}^{\infty}   n^{2m-1}  \bigl\vert
\Delta_{n-1}^{m}(x)\bigr\vert^2\,.
$$
Applying (\ref{3DiscreteSF2}) twice, with $m$ and $m+1$, we deduce the estimate 
(\ref{3DiscreteSF3}).

\smallskip (3)
Set 
\begin{equation}\label{6Stein}
M_n(T)=\,\frac{1}{n+1}\,\sum_{k=0}^{n} T^{k}
\end{equation}
for any $n\geq 0$. Then we have
\begin{equation}\label{3Stein}
\Bignorm{\Bigl(\sum_{n=0}^{\infty}(n+1)\bigl\vert M_{n+1}(T)(x) - M_{n}(T)(x) \bigr\vert^2\Bigr)^{\frac{1}{2}}}_p\,\lesssim \, \norm{x}_p
\end{equation}
for $x\in L^p(\Omega)$. By an entirely classical averaging
argument, one obtains this estimate as a consequence of (\ref{3DiscreteSF2}). We skip the details.

Inequality (\ref{3Stein}) plays a key role in \cite[Section 5]{S2}, where it is shown in the case when 
$T$ acts as a contraction $L^q(\Omega)\to L^q(\Omega)$ for any $1\leq q\leq\infty$ and its $L^2$-realization
is a positive selfadjoint operator.
\end{remark}

\begin{remark}\label{3Remark2} For any $\gamma\in\bigl(0,\frac{\pi}{2}\bigr)$, 
let $\P_\gamma\subset C(B_\gamma)$ be the algebra 
$\P$ regarded as a subspace of $C(B_\gamma)$, the commutative $C^*$-algebra of all complex valued continuous 
functions on the compact set $B_\gamma$. Let $u_\gamma\colon \P_\gamma\to B(L^p(\Omega))$ be the natural
functional calculus map, defined by 
$$
u_\gamma(F)=F(T).
$$

\smallskip
(a) Proposition \ref{3Matrix} means that for some $\gamma\in\bigl(0,\frac{\pi}{2}\bigr)$, 
the map $u_\gamma$ is $\ell_2$-completely bounded in the sense
of \cite{Sim} (see also \cite[Section 4]{KL}). In the case $p=2$, this means that $u_\gamma$ is completely
bounded. 

\smallskip
(b) If we restrict (\ref{3Matrix1}) to diagonal matrices, we readily obtain that whenever $(F_n)_{n\geq 1}$ is a bounded
sequence of $\P_\gamma$, then the set 
$\{F_n(T)\,: n\geq 1\}$ is $R$-bounded. Applying this property 
to the two sequences
$$
z\mapsto z^n\qquad \hbox{and}\qquad z\mapsto n(z^{n}-z^{n-1}),
$$
we deduce that any analytic contractively regular $T\colon L^p(\Omega)\to L^p(\Omega)$ is an
$R$-analytic power bounded operator (in the sense of Section 2).
This result is due to Blunck (see \cite[Thm. 1.1 and Thm. 1.2]{Bl1}).
\end{remark}

\medskip
\section{Maximal theorems on $L^p(\Omega)$.}
The general maximal theorem we aim at proving is the following.
The case $m=0$, which gives (\ref{1Max0}), is of particular interest.

\begin{theorem}\label{4Main}
Let $T\colon L^p(\Omega)\to L^p(\Omega)$ be an analytic 
contractively regular operator.
Then for any integer $m\geq 0$, there is a constant $C\geq 0$ such that 
\begin{equation}\label{4Max}
\Bignorm{\sup_{n\geq 0}\, (n+1)^m\bigl\vert T^n(T-I)^m(x)\bigr\vert}_p\,\leq\, C\,
\norm{x}_p,\qquad x\in L^p(\Omega).
\end{equation}
\end{theorem}

\begin{proof}
We will use classical `integration by parts' arguments and induction. Recall the notation from 
(\ref{3Delta}). 
For any $m\geq 1$, 
let us consider the estimate
\begin{equation}\label{4WeakMax}
\Bignorm{\sup_{n\geq 0}\,\frac{1}{n+1}\,\Bigl\vert  
\sum_{k=0}^{n} B_{k}^{m}(x)\Bigr\vert}_p\,\lesssim\,\norm{x}_p,\qquad x\in L^p(\Omega).
\end{equation}
This is clearly weaker than (\ref{4Max}), however we will need to use it explicitly later on.
For clarity we will write (\ref{4Max})$_m$ and (\ref{4WeakMax})$_m$ instead of (\ref{4Max})
and (\ref{4WeakMax}) in this proof.

For any $n\geq 1$, we have
$$
\sum_{k=1}^{n} k\bigl(T^k -T^{k-1}\bigr)\, =\,\sum_{k=1}^{n} kT^{k}\, -\,\sum_{k=0}^{n-1} (k+1) T^{k}\,  =\, n T^n 
\,-\,\sum_{k=0}^{n-1} T^{k}\,,
$$
hence
\begin{equation}\label{4IPP1}
T^n\,=\,
\frac{1}{n}\,\sum_{k=0}^{n-1} T^{k}\, +\, \frac{1}{n}\,\sum_{k=1}^{n} k\bigl(T^{k} - T^{k-1}\bigr)\,.
\end{equation}
By Cauchy-Schwarz, we deduce that for any $x\in L^p(\Omega)$,
$$
\bigl\vert T^n(x)\bigr\vert\,  \leq\,
\frac{1}{n}\,\Bigl\vert \sum_{k=0}^{n-1} T^{k}(x)\Bigr\vert
\, +\,
\Bigl(\sum_{k=1}^{n} k\bigl\vert T^{k}(x) - T^{k-1}(x)\bigr\vert^{2}\Bigr)^{\frac{1}{2}}.
$$
According to \cite{Pe} or \cite{CRW} (which generalized Akcoglu's Theorem to contractively regular 
operators), $T$ satisfies (\ref{1Ergodic}). Hence applying (\ref{3DiscreteSF2}) with $m=1$, we obtain (\ref{4Max})$_0$. Appealing to (\ref{4IPP1}) again, we immediatly deduce that
(\ref{4WeakMax})$_1$ holds true as well.

Now let $m\geq 1$.
Arguing as above we have 
\begin{equation}\label{4IPP2}
B_{n}^{m} \,=\,\frac{1}{n}\,\sum_{k=0}^{n-1} B_{k}^{m}\, +\, 
\frac{1}{n}\,\sum_{k=1}^{n} k\bigl(B_{k}^{m} - B_{k-1}^{m}\bigr).
\end{equation}
Also we have
\begin{align*}
\sum_{k=0}^{n} B_{k}^{m+1} \, &=\, (T-I)^{m}\,\sum_{k=0}^{n} (k+1)^{m+1}\bigl(T^{k+1} -T^{k}\bigr)\\
&=\, (T-I)^{m}\,\Bigl((n+1)^{m+1} T^{n+1}\,-\, \sum_{k=0}^{n} \bigl((k+1)^{m+1} -k^{m+1}\bigr) T^k\Bigr),
\end{align*}
hence
\begin{equation}\label{4IPP3}
\frac{1}{n+1}\,\sum_{k=0}^{n} B_k^{m+1}\,=\,\Bigl(\frac{n+1}{n+2}\Bigr)^{m+1} B_{n+1}^{m}\, -\,
\frac{1}{n+1}\,\sum_{k=0}^{n}\bigl((k+1)^{m+1} -k^{m+1}\bigr) \Delta_{k}^{m} .
\end{equation}
By Cauchy-Schwarz,
$$
\frac{1}{n}\,\Bigl\vert \sum_{k=1}^{n} k\bigl(B_{k}^{m}(x) - B_{k-1}^{m}(x)\bigr)\Bigr\vert\,\leq\,
\Bigl(\sum_{k=1}^{n} k \bigl\vert B_{k}^{m}(x) - B_{k-1}^{m}(x) \bigr\vert^2\Bigr)^{\frac{1}{2}},
$$
hence (\ref{4WeakMax})$_m$ implies (\ref{4Max})$_m$ by (\ref{4IPP2}) and Remark \ref{3Remark1}. Likewise,
\begin{align*}
\frac{1}{n+1}\,\Bigl\vert\sum_{k=0}^{n}\bigl((k+1)^{m+1} & -k^{m+1}\bigr)\Delta_{k}^{m}(x)
\Bigr\vert\\ 
& \leq\,\frac{1}{n+1}\,\biggl(\sum_{k=0}^{n}\frac{\bigl((k+1)^{m+1} -k^{m+1}\bigr)^2}{(k+1)^{2m-1}}\biggr)^{\frac{1}{2}}
\,\Bigl( \sum_{k=0}^{n} (k+1)^{2m-1}\bigl\vert\Delta_{k}^{m}(x)\bigr\vert^2\Bigr)^{\frac{1}{2}}\\
& \lesssim \Bigl(\sum_{k=0}^{n} (k+1)^{2m-1}\bigl\vert \Delta_{k}^{m}(x)\bigr\vert^2 \Bigr)^{\frac{1}{2}},
\end{align*}
hence 
(\ref{4WeakMax})$_{m+1}$ and (\ref{4Max})$_{m}$ are equivalent by (\ref{4IPP3}) and (\ref{3DiscreteSF2}). Thus 
(\ref{4Max})$_m$ holds true for any $m\geq 0$ by induction.
\end{proof}

%
%
%
%
%

Theorem \ref{4Main} is a generalization of \cite{S1}. 
In that paper, (\ref{4Max}) is established for an operator $T$ 
which is a positive contraction $L^q(\Omega)\to L^q(\Omega)$ for any 
$1\leq q\leq\infty$ whose $L^2$-realization is a positive selfadjoint operator. 
Clearly the $L^p$-realization of such an operator satisfies the assumptions of Theorem \ref{4Main}.
Indeed if $T\colon L^2(\Omega)\to L^2(\Omega)$ is a positive selfadjoint operator, then 
it is analytic by spectral representation. Hence $T\colon L^p(\Omega)\to L^p(\Omega)$ is analytic for 
any $1<p<\infty$ by \cite[Thm 1.1]{Bl2}.

The following is an analog of Theorem \ref{4Main} for continuous semigroups.

\begin{corollary}\label{4Sgp}
Let $(T_t)_{t\geq 0}$ be a bounded analytic semigroup on $L^p(\Omega)$, and assume
that $\rnorm{T_t}\leq 1$ for any $t\geq 0$. Then for any integer $m\geq 0$, we have an estimate
\begin{equation}\label{4IMm}
\biggnorm{\sup_{t>0} t^m\Bigl\vert\frac{\partial^m}{\partial t^m}\,\bigl(T_t(x)\bigr)\Bigr\vert}_p\,
\lesssim \, C\norm{x}_p,\qquad 
x\in L^p(\Omega).
\end{equation}
\end{corollary}

\begin{proof}
Let $-A$ be the generator of 
$(T_t)_{t\geq 0}$. According to Proposition \ref{2Weis}, it admits a bounded
$H^{\infty}(\Sigma_{\theta_0})$ functional calculus for some $\theta_0<\frac{\pi}{2}$. 
Let $\theta\in\bigl(\theta_0,\frac{\pi}{2}\bigr)$. Arguing as in 
Proposition \ref{3Matrix} and Theorem \ref{3DiscreteSF} (1), we obtain the
existence of a constant $C\geq 1$
such that for any sequence $(f_n)_{n\geq 1}$ of functions in 
$H^{\infty}_0(\Sigma_{\theta_0})$ and any $x\in L^p(\Omega)$, we have
$$
\Bignorm{\Bigl(\sum_{n=1}^{\infty}\bigl\vert f_n(A)x\bigr\vert^2\Bigr)^{\frac{1}{2}}}_p\,
\leq\, C\,\norm{x}_p\,\sup\Bigl\{ 
\Bigl(\sum_{n=1}^{\infty}\bigl\vert f_n(z)\bigr\vert^2\Bigr)^{\frac{1}{2}}\, :\, z\in \Sigma_\theta\Bigr\}.
$$
Then arguing as in Theorem \ref{3DiscreteSF} (2), we deduce that for any $m\geq 1$,
there is a constant $C_m\geq 1$ such that for any $t>0$ and for any 
$x\in L^p(\Omega)$,
$$
\Bignorm{\Bigl(\sum_{n=0}^{\infty} (n+1)^{2m-1}\bigl\vert T_t^{n}(T_t-I)^m(x)
\bigr\vert^2\Bigr)^{\frac{1}{2}}}_p\,\leq\, C_m\,\norm{x}_p.
$$
In other words, the operators $T_t$ satisfy (\ref{3DiscreteSF2}) uniformly.
The above proof of Theorem \ref{4Main} therefore shows that they satisfy
(\ref{4Max}) uniformly.

Let $t_1, t_2,\ldots, t_N$ be positive real numbers. For any $j=1,\ldots,N$ and $k\geq 1$, let
$n_{jk}$ be the integral part of $kt_j+1$ and let $t_{jk}=n_{jk}/k$, so that $t_{kj}\geq t_j$
and $t_{kj}\to t_j$ when $k\to\infty$. It follows from above that
we have an estimate
$$
\Bignorm{\Bigl( (n_{jk}+1)^m T_{\frac{1}{k}}^{n_{jk}}\bigl(T_{\frac{1}{k}} -I\bigr)^m (x)\Bigr)_{1\leq j\leq N}}_{L^p(\Omega;\ell^{\infty}_N)}\, \leq \, K\norm{x}_p,
$$
for some constant $K\geq 1$ neither depending on $x$, $k$ or the $t_j$'s. Letting $k\to \infty$, we deduce that
$$
\Bignorm{\Bigl( t_j^m  (-A)^m T_j(x)\Bigr)_{1\leq j\leq N}}_{L^p(\Omega;\ell^{\infty}_N)}\, \leq \, K\norm{x}_p.
$$
Clearly this uniform estimate implies (\ref{4IMm}).
\end{proof}

\begin{remark} 
Here is an alternative proof of Corollary \ref{4Sgp} not using the discrete case.
For any real $t>0$, consider the average operator $M_t\in B(L^p(\Omega))$ defined by letting
$$
M_t(x)\,=\,\frac{1}{t}\int_{0}^{t} T_u(x)\, du
$$
for any $x\in L^p(\Omega)$.
Since $\rnorm{T_t}\leq 1$ for any $t\geq 1$, 
it follows from \cite{Fe} that we have an estimate 
\begin{equation}\label{3Fendler}
\bignorm{\sup_{t>0}\bigl\vert M_t(x)\bigr\vert}_p\,\lesssim\,\norm{x}_p,\qquad x\in L^p(\Omega).
\end{equation}
For any integer $m\geq 1$, let $\varphi_m$ be the analytic function defined by $\varphi_m(z)=z^me^{-z}$. Then 
$\varphi_m$ belongs to $\HI_0(\Sigma_\theta)$ for any $\theta\in\bigl(0,\frac{\pi}{2}\bigr)$. Hence according to 
Propositions \ref{2SFE} and \ref{2Weis}, the square function estimate (\ref{2SFEbis}) 
holds for $\varphi=\varphi_m$.
For any real $t>0$, we have
$$
\varphi_m(tA)(x) =  t^mA^me^{-tA}(x) = (-1)^m t^m\,\frac{\partial^m}{\partial t^m}\,\bigl(T_t(x)\bigr).
$$
Hence we obtain estimates 
$$
\Bignorm{\Bigl(\int_{0}^{\infty} t^{2m-1} \Bigl\vert 
\frac{\partial^m}{\partial t^m}\,\bigl(T_t(x)\bigr)
\Bigr\vert^2\, \frac{dt}{t}\,\Bigr)^{\frac{1}{2}}}_p\, \lesssim\, \norm{x}_p
$$
for any $m\geq 1$. Then Stein's arguments in \cite[pp. 73-76]{S2} show that 
(\ref{3Fendler}) together with these estimates imply that (\ref{4IMm}) holds true for any $m\geq 0$.
\end{remark}

\medskip
\section{Maximal theorems on noncommutative $L^p$-spaces}
In this section we will partly extend the results established 
in the previous one, in the light of the recent work \cite{JX}.
We start with a few preliminaries on semifinite 
noncommutative $L^p$-spaces. 

Let $M$ be a von Neumann algebra equipped with a normal semifinite faithful trace $\tau$.
Let $M_+$ be the set of all positive elements of $M$ and let 
$S_+$ be the set of all $x$ in $M_+$ such that 
$\tau(x)<\infty$. Then let $S$ be the linear span of $S_+$. For any $1\leq p<\infty$, define 
$$
\norm{x}_p\,=\,\bigl(\tau(\vert x\vert^p)\bigr)^{\frac{1}{p}},\qquad x\in S,
$$
where $\vert x\vert =(x^*x)^{\frac{1}{2}}$ is the modulus of $x$. 
Then $(S,\norm{\ }_p)$ is a normed space. The corresponding 
completion is the noncommutative $L^p$-space associated with $(M,\tau)$ and is
denoted by $L^p(M)$. By convention we set $L^\infty(M)=M$, equipped with the operator norm.
The elements of $L^p(M)$ can also be described as measurable operators with respect to
$(M,\tau)$. Further multiplication of measurable operators leads to contractive
bilinear maps $L^{p}(M)\times L^{q}(M)\to L^{r}(M)$ for any $p,q,r$ such that $p^{-1} + q^{-1}= r^{-1}$ 
(noncommutative H\"older's inequality). Using trace duality, we then have 
$L^{p}(M)^*\,=\, L^{p'}(M)$ isometrically for any $1\leq p<\infty$. Moreover, complex interpolation
yields
\begin{equation}\label{5Inter}
L^{p}(M)\,=\, [L^\infty(M), L^1(M)]_{\frac{1}{p}}
\end{equation}
for any $1\leq p< \infty$. We refer the reader to \cite{PX} for details and complements.

Maximal functions in the noncommutative setting require a specific definition. Indeed,
$\sup_{n}\vert x_n\vert$ does not make any sense for a sequence $(x_n)_n$ of operators.
This difficulty is overcome by considering the spaces 
$L^p(M;\ell^\infty)$, which are the noncommutative analogs of
the usual Bochner spaces $L^p(\Omega;\ell^\infty)$. Given $1\leq p<\infty$, 
$L^p(M;\ell^\infty)$ is defined as the space of all sequences
$(x_n)_{n\geq 0}$ in $L^p(M)$ for which there exist $a,b\in L^{2p}(M)$ and a bounded 
sequence $(z_n)_{n\geq 0}$ in $M$ such that 
\begin{equation}\label{5Factor}
x_n=az_nb,\qquad n\geq 0. 
\end{equation}
For such a sequence, set
$$
\bignorm{(x_n)_{n\geq 0}}_{L^p(M;\ell^\infty)}\, =\, \inf\bigl\{\norm{a}_{2p}\sup_n\norm{z_n}\norm{b}_{2p}\},
$$
where the infimum runs over all possible factorizations of 
$(x_n)_{n\geq 0}$ in the form (\ref{5Factor}). This is a norm and $L^p(M;\ell^\infty)$ is a Banach space.
These spaces were first introduced by Pisier \cite{P} in the case when $M$ is hyperfinite and by Junge \cite{J} in the general case. We will adopt the convention in \cite{JX} that the norm 
$\norm{(x_n)_{n\geq 0}}_{L^p(M;\ell^\infty)}$ is denoted by
\begin{equation}\label{5Sup}
\bignorm{{\sup_{n\geq 0}}^{+} x_n}_p.
\end{equation}
We warn the reader that this suggestive notation
should be treated with care.
It is used for possibly non positive operators and
$\bignorm{{\displaystyle{\sup_{n\geq 0}}^{+} x_n}}_p\not=\bignorm{\displaystyle{{\sup_{n\geq 0}}^{+} \vert x_n\vert}}_p$
in general. However it has an intuitive description in the positive case, as
observed in \cite[p. 392]{JX}: a positive sequence 
$(x_n)_{n\geq 0}$ of $L^p(M)$ belongs to 
$L^p(M;\ell^\infty)$ if and only if there exists a positive 
$a\in L^p(M)$ such that $x_n\leq a$ for any $n\geq 0$ and in this case,
\begin{equation}\label{5Domin}
\bignorm{{\sup_{n\geq 0}}^{+} x_n}_p\,=\,\inf\bigl\{\norm{a}_p\, :\, a\in L^p(M),\ a\geq 0
\quad\hbox{and}\quad x_n\leq a\ \hbox{for any } n\geq 0\bigr\}.
\end{equation}

\bigskip
Let $T\colon M\to M$ be a contraction
We say that it is an absolute contraction if  
its restriction to $L^1(M)\cap M$ extends to a contraction
$L^1(M)\to L^1(M)$. In this case, it extends (by interpolation)
to a contraction on $L^p(M)$ for any $1\leq p\leq \infty$. We let
$T_p\colon L^p(M)\to L^p(M)$ denote the resulting operator.

\begin{lemma}\label{5Analytic} Let $1<p,q<\infty$. The operator 
$T_p$ is analytic if and only if $T_q$ is analytic.
\end{lemma}

\begin{proof} This result was proved by Blunck in the commutative setting
\cite[Thm. 1.1]{Bl2}, using interpolation. 
His arguments apply as well to the noncommutative setting, using (\ref{5Inter}).
\end{proof}

In accordance with this lemma we will say that an absolute
contraction $T\colon M\to M$ is analytic if
$T_p$ is analytic for one (equivalently for all) $1<p<\infty$.

We say that $T\colon M\to M$ is positive if $T(x)\geq 0$ for any $x\in M_+$. 
If $T$ is an absolute contraction, then $T_p(x)\geq 0$ for any 
$x\in L^p(M)_+$ and any $p$. 

\begin{theorem}\label{5Main}
Let $T$ be a positive analytic absolute contraction. Then for any
$1<p<\infty$ and any integer $m\geq 0$, we have an estimate
\begin{equation}\label{5IMm}
\Bignorm{{\sup_{n\geq 0}}^{+} (n+1)^m T^n(T-I)^m (x)}_p\,\lesssim\, \norm{x}_p,\qquad
x\in L^p(M).
\end{equation}
\end{theorem}

In particular we obtain a maximal inequality
$$
\Bignorm{{\sup_{n\geq 0}}^{+} T^n(x)}_p\,\lesssim\, \norm{x}_p
$$
for any $T$  as above.

These maximal theorems were proved in \cite{JX} under the assumption that the Hilbertian operator
$T_2\colon L^2(M)\to L^2(M)$ is selfadoint and positive in the sense that 
$\sigma(T_2)\subset [0,1]$. This was recently extended by Bekjan \cite{Bek} to the case when 
the numerical range of $T_2$ is included in a Stolz domain $B_\gamma$ 
for some $\gamma\in\bigl(0,\frac{\pi}{2}\bigr)$. These results are covered 
by Theorem \ref{5Main}. Indeed it is easy to see that the latter numerical range
condition implies that 
$T_2$ is analytic.

A key step in proving Theorem \ref{5Main} is the following series of square function estimates.

\begin{proposition}\label{5SFE}
Let $T\colon M\to M$ be an analytic absolute contraction. Then for any integer 
$m\geq 1$, we have an estimate
\begin{equation}\label{5SFE1}
\Bigl(\sum_{n=0}^{\infty} (n+1)^{2m-1}\bignorm{T^n(T-I)^m(x)}^2_2
\Bigr)^{\frac{1}{2}}\,\lesssim\,\norm{x}_2,\qquad x\in L^2(M).
\end{equation}
\end{proposition}

\begin{proof}
The argument is entirely similar to the one devised to prove (\ref{3DiscreteSF2}). 
We use the assumption that $T_2$ is analytic. We let
$A=I-T_2$ and we let $(T_t)_{\geq 0}$ be the semigroup generated by $-A$ on $L^2(M)$.
This is a bounded analytic semigroup and since $T_2$ is a contraction, 
we have $\norm{T_t}\leq 1$ for any $t\geq 0$. Hence by \cite{MI} (see also \cite{LM2}),
$A$ admits a bounded $\HI(\Sigma_{\theta_0})$ for some $\theta_0<\frac{\pi}{2}$ and
hence, for every $\theta\in\bigl(\theta_0,\frac{\pi}{2}\bigr)$ and
for any $\varphi\in\HI_0(\Sigma_\theta)$, there exists a constant $C\geq 0$ such that
\begin{equation}\label{5SFEHilbert}
\Bigl(\int_{0}^{\infty}\bignorm{\varphi(tA)x}_2^2\, \frac{dt}{t}\,\Bigr)^{\frac{1}{2}}\,\leq\, C\norm{x}_2,
\qquad x\in L^2(M).
\end{equation}
Arguing as in the proof of Proposition \ref{3Matrix} and using (\ref{5SFEHilbert}) 
in place of Proposition \ref{2SFE}, we obtain that there exists an angle 
$\gamma\in \bigl(0,\frac{\pi}{2}\bigr)$ such that the natutal functional calculus
$$
u_\gamma\colon \P_\gamma\longrightarrow B(L^2(M)),\qquad u_\gamma(F)=F(T_2),
$$
is completely bounded. That is, there exists a 
constant $C\geq 0$ such that for  any $N\geq 1$, for any $N\times N$ matrix $[F_{ij}]$ of polynomials 
and for any
$x_1,\ldots, x_N$ in $L^2(M)$, 
$$
\Bigl(\sum_{i=1}^{N}\Bignorm{\sum_{j=1}^{N} F_{ij}(T)x_j}_2^2\Bigr)^{\frac{1}{2}} 
\,\leq\, C
\bignorm{[F_{ij}]}_\gamma\, \Bigl(\sum_{j=1}^{N} \norm{x_j}_2^2\Bigr)^{\frac{1}{2}}.
$$
Then the argument in the proof of Theorem \ref{3DiscreteSF} yields the result.
\end{proof}

\begin{proof}[Proof of Theorem \ref{5Main}]
Once we have  the estimates (\ref{5SFE1}) in hands, one can 
deduce Theorem \ref{5Main} by repeating the arguments of \cite[Section 5]{JX} 
(see also \cite{Bek}). 
\end{proof}

\begin{remark}
Let $T$ be as in Theorem \ref{5Main} and for any complex number $\alpha$, let
$M_n^{\alpha}(T)$ be defined as in \cite[p. 409]{JX}. ($M_n^1(\cdotp)$ is equal to the average
$M_n(\cdotp)$ given by (\ref{6Stein}).) Then the argument in \cite[Section 5]{JX}
shows that for any $\alpha\in\Cdb$ and any $1<p<\infty$, there is an estimate
$$
\Bignorm{{\sup_{n\geq 0}}^{+} M_n^\alpha(T)x}_p\,\lesssim\, \norm{x}_p,\qquad 
x\in L^p(M).
$$
The estimate (\ref{5IMm}) corresponds to $\alpha=-m$.

A similar comment applies to Theorem \ref{4Main}.
\end{remark}

Following \cite[Rem. 2.4]{JX}, the definition of $L^p(M,\ell^\infty)$ can be extended 
to arbitrary index sets. For any set $I$ and any $1\leq p<\infty$, 
$L^p(M;\ell^\infty_I)$ is defined as the space of all families $(x_i)_{i\in I}$ of 
$L^p(M)$ which can be factorized as $x_i=az_ib$, where $a,b\in L^{2p}(M)$ and
$(z_i)_{i\in I}$ belongs to $\ell^\infty_I(M)$. Moreover the norm of 
$(x_i)_{i\in I}$ in $L^p(M;\ell^\infty_I)$ is defined as the infimum of all
$\norm{a}_{2p}\sup_i\norm{z_i}\norm{b}_{2p}$ running over all such factorizations. 
We let $\bignorm{{\displaystyle{\sup_{i}}^{+} x_i}}_p$ denote the norm of an
element $(x_i)_{i\in I}$ of $L^p(M;\ell^\infty_I)$.
The analog of (\ref{5Domin}) holds in this general case, that is, 
a positive family $(x_i)_{i\in I}$ belongs to 
$L^p(M;\ell^\infty_I)$ if and only if there exists a positive $a\in L^p(M)$ such that 
$x_i\leq a$ for any $i\in I$ and moreover,
\begin{equation}\label{6Domin}
\bignorm{{\sup_{i}}^{+} x_i}_p\,=\,\inf\bigl\{\norm{a}_p\, :\, a\in L^p(M),\ a\geq 0
\quad\hbox{and}\quad x_i\leq a\ \hbox{for any } i\in I\bigr\}.
\end{equation}
In the sequel we will deal with semigroups and apply the above facts with $I=\Rdb_+$.

Let $(T_t)_{t\geq 0}$ be a semigroup of operators on $M$.
Assume that for any $t\geq 0$, $T_t$ is an absolute contraction and that 
for any $1< p<\infty$, $(T_t)_{t\geq 0}$ is strongly continuous
on $L^p(M)$. (By \cite[Prop. 1.23]{D}, this holds true for example if
for any $x\in M$,
$T_t(x)\to x$ in the $w^*$-topology of $M$ when $t\to 0^+$.) 
We let $-A_p$ denote the generator 
of $(T_t)_{t\geq 0}$ acting on $L^p(M)$. 

Given any two indices $1<p,q<\infty$, $A_p$ is sectotial of type $<\frac{\pi}{2}$
if and only if $A_q$ is sectotial of type $<\frac{\pi}{2}$. In other words 
$(T_t)_{t\geq 0}$ being a bounded analytic semigroup on $L^p(M)$ does not depend on $1<p<\infty$. 
This is a continuous analog of Lemma \ref{5Analytic}, whose proof is identical 
to the one of \cite[Prop. 5.4]{JLX}. We skip the details.

\begin{theorem}\label{6IMSG}
Let $(T_t)_{t\geq 1}$ be a semigroup on $M$ as above.
Assume that for any $t\geq 0$,
$T_t$ is positive and that for one $1<p<\infty$ (equivalently, for all
$1<p<\infty$),
$(T_t)_{t\geq 1}$ is analytic on $L^p(M)$. 
Then for any $1<p<\infty$ and any integer $m\geq 0$,
we have an estimate
$$
\biggnorm{{\sup_{t>0}}^{+} t^m  \,
\frac{\partial^m}{\partial t^m}\,\bigl(T_t(x)\bigr)}_p\,\lesssim\,\norm{x}_p,\qquad x\in L^p(M).
$$
\end{theorem}

\begin{proof}
Fix $p$ and $m\geq 0$. According to \cite[Prop. 2.1 and Rem. 2.4]{JX}, it suffices to find a constant
$C\geq 0$ such that for any finite family $t_1,\ldots,t_N$ of positive real numbers,
$$
\biggnorm{{\sup_{k}}^{+} t_k^m  \,
\frac{\partial^m}{\partial t^m}\,\bigl(T_t(x)\bigr)_{\big\vert t=t_k}}_p\,\leq\, C\,\norm{x}_p
$$
for any $x\in L^p(M)$. 
This follows from Theorem \ref{5Main}, using the same approximation argument as in 
the proof of Corollary \ref{4Sgp}.
\end{proof}

We end this section with applications to $R$-analyticity (see Section 2 for terminology and background).
We recall Weis's Theorem \cite{W2} that if $(T_t)_{t\geq 0}$ is a bounded analytic 
semigroup on some commutative $L^p$-space (with $1<p<\infty$) 
such that each $T_t$ is contractively regular, then 
$(T_t)_{t\geq 0}$ is actually an $R$-bounded analytic semigroup. 
(This result was used in the proof of
Proposition \ref{2Weis} in the present paper.) 
The next corollary is an analog of that result in our noncommutative setting. 
In the selfadjoint case, it was established in \cite[Thm. 5.6]{JLX}.
The proof in the analytic case follows a similar scheme so we will be brief.

\begin{corollary}\label{6R1} Let $(T_t)_{t\geq 1}$ be as in Theorem \ref{6IMSG}. 
Then for any $1<p<\infty$, the realization of $(T_t)_{t\geq 0}$ on $L^p(M)$ is 
an $R$-bounded analytic semigroup.
\end{corollary}

\begin{proof} We first observe that 
the dual semigroup $(T_t^*)_{t\geq 0}$ satisfies the assumptions of 
Theorem \ref{6IMSG}. Let $1<r<\infty$. Applying the latter theorem for $m=0$ and (\ref{6Domin}),
we find a constant $C_r>0$ such that for any $y\in L^r(M)_+$, there exists $a\in L^r(M)_+$ such that
$$
\norm{a}_r\leq C_r\norm{y}_r\qquad\hbox{and}\qquad T_t^*(y)\leq a\quad\hbox{for any }\ t\geq 0.
$$
Then the argument in the proof of \cite[Thm. 5.6]{JLX} shows that for any $2\leq q<\infty$, the set
\begin{equation}\label{6Fq}
F_q=\bigl\{ T_t\colon L^q(M)\longrightarrow L^q(M)\, :\, t\geq 0\bigr\}
\end{equation}
is $R$-bounded. 

The analyticity assumption ensures the existence of an angle
$\nu\in\bigl(0,\frac{\pi}{2}\bigr)$ 
such that the realization 
of $(T_t)_{t\geq 0}$ on $L^2(M)$ extends to a 
bounded family $(T_z)_{z\in\overline{\Sigma_\nu}}$ of opertors on $L^2(M)$, 
whose restriction to $\Sigma_\nu$ is analytic.
Since boundedness is equivalent to $R$-boundedness on Hilbert spaces, this immediately implies that the sets
\begin{equation}\label{6G2}
\bigl\{ T_{te^{i\nu}}\colon L^2(M)\to L^2(M)\, :\, t\geq 0\bigr\}\qquad\hbox{and}\qquad
\bigl\{ T_{te^{-i\nu}}\colon L^2(M)\to L^2(M)\, :\, t\geq 0\bigr\}
\end{equation}
are $R$-bounded.

Let $2<p<\infty$, let $q>p$ be a finite number and let 
$\alpha = 2(q-2)^{-1}\bigl(\frac{q}{p} -1\bigr)$. 
In accordance with (\ref{5Inter}), this
number is chosen 
so that $L^p(M)=[L^q(M),L^2(M)]_\alpha$. As is well-known, this implies 
that
$$
{\rm Rad}\bigl(L^p(M)\bigr)\, =\,\bigl[{\rm Rad}\bigl(L^q(M)\bigr),{\rm Rad}\bigl(L^2(M)\bigr)\bigr]_\alpha
$$
isomorphically. Applying Stein's interpolation principle as  
in the proof of \cite[Thm. 5.6]{JLX} and the $R$-boundedness of the
sets in (\ref{6G2}), we deduce that
$$
\bigl\{ T_z\colon L^p(M)\longrightarrow L^p(M)\, :\, z\in \Sigma_{\alpha\nu}\bigr\}
$$
is $R$-bounded. This shows that $(T_t)_{t\geq 0}$ is an
$R$-bounded analytic semigroup on $L^p(M)$.

The case $1<p<2$ easily follows by duality.
\end{proof}

Let us finally come back to the discrete case. Blunck \cite[Thm. 1.1 and Thm. 1.2]{Bl1} showed that
any analytic contractively regular operator on a commutative $L^p$-space (with $1<p<\infty)$ is an 
$R$-analytic power bounded operator (see Remark \ref{3Remark2} (b) in the present paper
for a proof of this result). 
This is a discrete analog of Weis's Theorem. Here is a noncommutative version.

\begin{proposition}\label{6R2} Let $T\colon M\to M$ be an absolute contraction
and assume that $T$ is positive. Let $1<p<\infty$. If $T$ is analytic, then 
$T_p\colon L^p(M)\to L^p(M)$ is an $R$-analytic power bounded operator for any $1<p<\infty$.
\end{proposition}

\begin{proof} Let $(T_t)_{t\geq 0}$ be defined by (\ref{2Tt}). 
Then for any $t\geq 0$, $T_t$ is a positive absolute contraction. Moreover for 
any $1<p<\infty$,  $(T_t)_{t\geq 0}$ is analytic on $L^p(M)$.
Hence by Proposition \ref{6R1}, $(T_t)_{t\geq 0}$ is actually an $R$-bounded
analytic  semigoup on $L^p(M)$. By \cite[Thm. 1.1]{Bl1}, this implies  that 
$T_p$  is an $R$-analytic power bounded operator.
\end{proof}

\begin{remark} 
Consider the notions of column boundedness and row boundedness as defined in \cite[Section 4.A]{JLX}
and let us state Col-bounded and Row-bounded versions of Theorem \ref{6IMSG} and 
Proposition \ref{6R2}.
Let $T\colon L^p(M)\to L^p(M)$ and let us say that $T$ is  Col-analytic 
(resp. Row-analytic) power bounded it the two sets $\P_T$ and $\A_T$ from (\ref{2Power}) and (\ref{2Analytic}) are both Col-bounded (resp. Row-bounded). Likewise, let us say that a semigroup 
$(T_t)_{t\geq 0}$ on $L^p(M)$ is a Col-bounded (resp. Row-bounded) analytic semigroup if
the two sets $\{T_t\, :\, t> 0\}$ and $\{tAT_t\, :\, t>0\}$ 
are both Col-bounded (resp. Row-bounded).

Let $(T_t)_{t\geq 0}$  be a semigroup
on $M$ as in Theorem \ref{6IMSG} and assume that $T_t\colon M\to M$
is 2-positive for any $t\geq 0$. Then as in \cite[Thm. 5.6]{JLX}, one can show 
that for any $1<p<\infty$,
the realization of $(T_t)_{t\geq 0}$ on $L^p(M)$ is 
both a Col-bounded and a Row-bounded analytic semigroup.

Likewise, if $T\colon M\to M$ is a 2-positive and analytic absolute contraction,
then for any $1<p<\infty$,
$T_p\colon L^p(M)\to L^p(M)$ is both a Col-analytic and a Row-analytic
power bounded operator. More concretely, this implies in particular that we have estimates
$$
\Bignorm{\Bigl(\sum_n  T^n(x_n)^* T^n(x_n)\Bigr)^{\frac{1}{2}}}_p\,\lesssim\,
\Bignorm{\Bigl(\sum_n  x_n^* x_n\Bigr)^{\frac{1}{2}}}_p
$$
and
$$
\Bignorm{\Bigl(\sum_n  T^n(x_n) T^n(x_n)^*\Bigr)^{\frac{1}{2}}}_p\,\lesssim\,
\Bignorm{\Bigl(\sum_n  x_n x_n^*\Bigr)^{\frac{1}{2}}}_p
$$
for any $1<p<\infty$.
\end{remark}

\end{document}